\documentclass[11pt,reqno]{amsart}

\usepackage{amsmath,amssymb,amsthm,mathrsfs}
\usepackage{array}
\usepackage[margin=1.15in]{geometry}
\usepackage{xcolor}
\usepackage[colorlinks=true,linkcolor=blue,citecolor=blue,urlcolor=blue]{hyperref}
\usepackage{pgffor}
\usepackage{tikz}

\definecolor{revcolor}{RGB}{0,125,0}
\definecolor{chkcolor}{RGB}{205,0,0}
\newif\ifshowchanges
\showchangestrue
\ifshowchanges

  \newenvironment{revblock}{\color{revcolor}}{}
  \newcommand{\authornote}[1]{{\color{chkcolor}[{Note for the authors:} #1]}}
\else

  \newcommand{\authornote}[1]{\ignorespaces}
\fi
\newif\ifmentionchen
\mentionchentrue

\theoremstyle{plain}
\newtheorem{theorem}{Theorem}[section]

\newtheorem{lemma}[theorem]{Lemma}
\newtheorem{corollary}[theorem]{Corollary}

\theoremstyle{definition}
\newtheorem{remark}[theorem]{Remark}
\newtheorem{example}[theorem]{Example}
\newtheorem{problem}[theorem]{Problem}

\numberwithin{equation}{section}

\newcommand{\R}{\mathbb{R}}

\newcommand{\Sph}{\mathbb{S}}
\newcommand{\dH}{\dot H}
\newcommand{\norm}[1]{\left\Vert #1\right\Vert}
\newcommand{\abs}[1]{\left\vert #1\right\vert}
\newcommand{\dd}{\,d}
\DeclareMathOperator{\Real}{Re}
\DeclareRobustCommand{\orcidicon}{%
  \begin{tikzpicture}
    \draw[lime, fill=lime] (0,0)
      circle [radius=0.16]
      node[white] {{\fontfamily{qag}\selectfont \tiny ID}};
    \draw[white, fill=white] (-0.0625,0.095) circle [radius=0.007];
  \end{tikzpicture}\hspace{-2mm}}
\foreach \x in {A, B}{%
  \expandafter\xdef\csname orcid\x\endcsname{%
    \noexpand\href{https://orcid.org/\csname orcidauthor\x\endcsname}{\noexpand\orcidicon}}}

\begin{document}

\title[Exact $L^2$-growth profiles for the wave equation]
{Exact $L^2$-growth profiles for the free wave equation:\\
sharp constants, critical regularity, and local energy decay}

\author[V.~Kumar and M.~Majdoub]{Vishvesh Kumar\orcidA{} and Mohamed Majdoub\orcidB{}}

\address[V.~Kumar]{
Department of Mathematical Sciences, Indian Institute of Technology (BHU),
Varanasi, Uttar Pradesh, 221005, India.}
\email{\tt vishvesh.mat@iitbhu.ac.in}
\email{\tt vishveshmishra@gmail.com}

\address[M.~Majdoub]{Department of Mathematics, College of Science, Imam Abdulrahman Bin Faisal University, P.O.~Box 1982, Dammam, Saudi Arabia.\newline
Basic and Applied Scientific Research Center, Imam Abdulrahman Bin Faisal University, P.O.~Box 1982, 31441, Dammam, Saudi Arabia.}
\email{\tt mmajdoub@iau.edu.sa}
\email{\tt  med.majdoub@gmail.com}
\email{\tt mohamed.majdoub@fst.rnu.tn}
\begin{abstract}
We study the large-time behavior of the $L^2$-norm of solutions to the free wave equation on $\mathbb{R}^N$ and, more generally, their homogeneous Sobolev norms of negative order. We refine and extend the results of R.~Ikehata [J.~Hyperbolic Differ.\ Equ.\ 20 (2023)]{, who identified and quantified the} infinite-time $L^2$-growth {of free waves} in dimensions $N=1,2$. For finite-energy data with $u_1\in L^1(\mathbb{R}^N)$, we derive exact first-order asymptotics: when $P=\int_{\mathbb{R}^N}u_1(x)\,dx\neq0$, the squared $L^2$-norm grows linearly for $N=1$ and logarithmically for $N=2$, with a leading coefficient depending only on $P$. In particular, no weighted moment assumption on $u_1$ is required; $u_1\in L^1\cap L^2$ suffices.

More generally, we prove that the $L^2$-norm is bounded, and hence convergent, if and only if $u_1\in\dot H^{-1}(\mathbb{R}^N)$. We further establish first-order asymptotic profiles in negative-order homogeneous Sobolev spaces, covering polynomial growth, logarithmic growth, and convergence to a finite limit according to the Sobolev index. These results reveal a unified low-frequency mechanism underlying the low-dimensional $L^2$-growth.

We also derive second-order asymptotic expansions in dimension one and at the critical Sobolev index in dimensions $N\ge2$. Finally, in dimension two, the exact $L^2$-asymptotics yield a sharper local energy decay estimate with an explicit upper asymptotic constant independent of the localization radius.
\end{abstract}
\maketitle

\section{Introduction}\label{sec:intro}

We consider the Cauchy problem for the free wave equation
\begin{equation}\label{eq:wave}
u_{tt}-\Delta u=0, \qquad (t,x)\in (0,\infty)\times\R^N,
\end{equation}
with initial conditions
\begin{equation}\label{eq:data}
u(0,x)=u_0(x),\qquad u_t(0,x)=u_1(x), \qquad x\in\R^N,
\end{equation}
where $[u_0,u_1]\in H^1(\R^N)\times L^2(\R^N)$ are real-valued functions. It is well known (see, for example, \cite{Racke1992,Ikawa2000}) that the Cauchy problem \eqref{eq:wave}--\eqref{eq:data} admits a unique weak solution
\begin{equation*}
u\in C\bigl([0,\infty);H^1(\R^N)\bigr) \cap C^1\bigl([0,\infty);L^2(\R^N)\bigr).
\end{equation*}
Moreover, the energy
\begin{equation}\label{eq:energy}
E(t) :=\frac12\Bigl( \|u_t(t,\cdot)\|_{L^2(\R^N)}^2 +\|\nabla u(t,\cdot)\|_{L^2(\R^N)}^2 \Bigr)
\end{equation}
is conserved along the flow, that is, $E(t)=E(0)$ for every $t\ge0$. Throughout the paper, $\|\cdot\|$ denotes the norm of $L^2(\R^N)$.

Energy conservation controls $u_t$ and $\nabla u$, but it gives no direct control on the $L^2$-norm of the solution itself. The behaviour of
$$
M(t):=\norm{u(t,\cdot)}
$$
is therefore a separate question. Unlike the energy, $M(t)$ is not conserved, and{, as discovered by Ikehata \cite{Ikehata2023}, in dimensions $N=1$ and $N=2$ it need not remain bounded.}

{The quantity $M(t)$ is not a mere curiosity. Estimates} for $M(t)$ arise naturally in the multiplier method for local energy decay, initiated by Morawetz \cite{Morawetz1961} and developed since then in various directions; see, for example, \cite{Ikehata2004a,Ikehata2005,IkeharaNishihara2005,IkehataSobukawa2007,AikawaIkehata2010,CharaoIkehata2020,Ikehata2019}.
To explain this connection, let
$$
E_R(t):=\frac12\int_{|x|\leq R} \left( |u_t(t,x)|^2+|\nabla u(t,x)|^2 \right)\,\dd x
$$
denote the local energy in the ball of radius $R>0$. A suitable multiplier argument typically leads to an estimate of the form
$$
tE_R(t)\lesssim K_0+|F(t)|,
$$
where $K_0$ depends only on the initial data and
$$
F(t):=\int_{\R^N}u_t(t,x)u(t,x)\,\dd x.
$$
The mixed term $F(t)$ involves both the velocity $u_t$ and the solution $u$. By the Cauchy--Schwarz inequality, we get
$$
|F(t)| \leq \norm{u_t(t,\cdot)}\norm{u(t,\cdot)} = \norm{u_t(t,\cdot)}M(t).
$$
Since $\norm{u_t(t,\cdot)}^2\leq 2E(t)=2E(0)$, we deduce that $|F(t)|\leq\sqrt{2E(0)}\,M(t)$, and therefore
$$
E_R(t) \lesssim \frac{K_0}{t} + \frac{\sqrt{2E(0)}}{t}\,M(t).
$$
Thus the decay rate of the local energy is directly influenced by the large-time behaviour of the global quantity $M(t)$, and any improvement in the understanding of $M$ yields an improvement of the local energy decay rate. {This connection is the starting point of \cite{Ikehata2023}. There, Ikehata proposed to study the large-time behaviour of the $L^2$-norm of the solution itself, proved that $M(t)$ blows up as $t\to\infty$ in dimensions one and two, with sharp rates, as soon as the total mass of the initial velocity is nonzero, and used this information to derive new local energy decay estimates; see Subsection~\ref{ss:ikehata}.} {To the best of our knowledge, the infinite-time $L^2$-growth of free waves in the whole space had not been explicitly identified and quantified before \cite{Ikehata2023}; see \cite{IkehataMatsuyama2002} for an earlier study of the $L^2$-behaviour of solutions in exterior domains.} {All the questions addressed in the present paper originate from \cite{Ikehata2023}.} The same quantity plays a central role for {damped,} strongly damped and structurally damped wave equations. There, the $L^2$-behaviour of the solution encodes the competition between the parabolic and the hyperbolic parts of the evolution; see \cite{Ikehata2014,IkehataOnodera2017,Ikehata2004} {and \cite{CharaoIkehata2022log}}.

\subsection{Ikehata's estimates}\label{ss:ikehata}

In \cite{Ikehata2023}, Ikehata {initiated the study of the large-time behaviour of $M(t)$ for the free wave equation}. He derived upper and lower bounds of the same order for the $L^2$-norm of solutions in dimensions $N=1$ and $N=2$. To state his result, we define, for $\gamma\geq0$, the weighted space
$$
L^{1,\gamma}(\R^N) := \left\{ f\in L^1(\R^N) : \norm{f}_{1,\gamma} := \int_{\R^N} \left(1+\abs{x}^{\gamma}\right)\abs{f(x)}\,\dd x <\infty \right\},
$$
and we set
$$
P:=\int_{\R^N}u_1(x)\,\dd x.
$$
The number $P$ is the zeroth moment, or total mass, of the initial velocity $u_1$. With the normalization of the Fourier transform used below, $P=\widehat{u}_1(0)$, so that $P$ measures the zero-frequency component of the initial velocity.

Let $[u_0,u_1]\in H^1(\R^N)\times L^2(\R^N)$ with $u_0,u_1\in L^1(\R^N)$, and set
$$
I_{0,N}:=\norm{u_0}+\norm{u_0}_{L^1(\R^N)}+\norm{u_1}+\norm{u_1}_{L^1(\R^N)}.
$$
For $N=1$, Ikehata proved \cite[Theorem~1.1]{Ikehata2023} that there exist constants $C_1,C_2>0$ such that
\begin{equation}\label{eq:ike1}
C_2\abs{P}\sqrt{t}\ \le\ M(t)\ \le\ C_1 I_{0,1}\sqrt{t},\qquad t\gg1,
\end{equation}
where the upper bound holds under the above assumptions, while the lower bound requires in addition $u_1\in L^{1,1}(\R)$. For $N=2$, under the same assumptions with $\R$ replaced by $\R^2$, he proved \cite[Theorem~1.2]{Ikehata2023} that
\begin{equation}\label{eq:ike2}
C_2\abs{P}\sqrt{\log t}\ \le\ M(t)\ \le\ C_1 I_{0,2}\sqrt{\log t},
\qquad t\gg1.
\end{equation}

In particular, as soon as the total mass $P$ of the initial velocity is nonzero, the solution blows up in $L^2$ as $t\to\infty$, at rate $\sqrt t$ in dimension one and $\sqrt{\log t}$ in dimension two. For $N\ge3$ {and data in $L^1\cap L^2$}, no such phenomenon occurs.

As an elementary special case, once the phenomenon discovered in \cite{Ikehata2023} is known, its one-dimensional version can be recovered, for compactly supported data, by an elementary physical-space computation. We include it because it makes the mechanism transparent. It should be regarded as an alternative simple proof of a special case of \cite[Theorem~1.1]{Ikehata2023}, found with the benefit of hindsight, and not as an earlier source of that result. The d'Alembert formula itself is classical (see, e.g., \cite[Section~2.4]{Evans2010}), but, as far as we know, the large-time behaviour of the $L^2$-norm of the solution is not discussed in the standard textbooks, including \cite{Evans2010}.

Assume that $N=1$, $u_0\equiv0$ and $u_1\in L^1(\R)\cap L^2(\R)$. Then
\begin{equation}\label{eq:dalembert}
u(t,x) =\frac12\int_{x-t}^{x+t}u_1(r)\,\dd r \longrightarrow \frac{P}{2} \qquad\text{as }t\to\infty,
\end{equation}
for every fixed $x\in\R$, where $P=\int_{\R}u_1(r)\,\dd r$ is the total mass introduced above. Hence, when $P\neq0$, the solution approaches the nonzero constant $P/2$ on a spatial region whose length grows linearly in time. If, in addition, $u_1$ vanishes outside $[-a,a]$ for some $a>0$, then, for $t\ge a$, we have $u(t,x)=P/2$ for $\abs{x}\le t-a$, $u(t,x)=0$ for $\abs{x}\ge t+a$, and $\abs{u(t,x)}\le\frac12\norm{u_1}_{L^1(\R)}$ on the two remaining intervals, each of length $2a$. Consequently,
$$
M(t)^2=\norm{u(t,\cdot)}^2 =\frac{P^2}{2}\,t+O(1) \qquad\text{as }t\to\infty,
$$
and therefore
$$
M(t)\sim\frac{|P|}{\sqrt{2}}\,t^{1/2}.
$$
This recovers \eqref{eq:ike1}, with the optimal constant in the leading term, but only for compactly supported velocities. For a general $u_1\in L^1(\R)\cap L^2(\R)$, the plateau is never reached exactly and the remainder need not be bounded; for the velocity of Example~\ref{ex:cauchy} below, the second-order term is $-\pi\log t$. In dimension two, the corresponding elementary argument, based on the conservation law \eqref{eq:massconservation} below, gives a positive lower bound for $M(t)$ but no growth at all. In both cases, the general result requires a genuinely different argument, such as the one of \cite{Ikehata2023} or the low-frequency analysis of Subsection~\ref{ss:mechanism}.

The growth of $M(t)$ is related to the failure of the strong Huygens principle in dimension $N=1$ and in even space dimensions, where the solution generally develops a tail inside the light cone. In low dimensions, the low-frequency component carried by the spatial mean of $u_1$ decays too slowly to allow a uniform-in-time $L^2$-bound. There is no contradiction with energy conservation. Indeed, the energy controls only $u_t$ and $\nabla u$, and a slowly varying component may have a large $L^2$-norm while contributing little to these two quantities.

Estimates \eqref{eq:ike1}--\eqref{eq:ike2} are sharp {in the rate}. They are not sharp in the constants, however. The upper and lower bounds differ by an unspecified factor and, more strikingly, they are of a different nature. The lower bound depends on the data only through the single number $\abs{P}$. In contrast, the upper bound involves the whole quantity $I_{0,N}$, {which contains the norms of $u_0$ and the full $L^1$-norm of $u_1$. Since $\abs{P}\le\norm{u_1}_{L^1}$, it may be much larger than $\abs{P}$.} One is left wondering which of the two bounds reflects the true behavior. This gap is the starting point of the present {paper}.

\subsection{Three questions}\label{ss:questions}
The preceding discussion naturally leads to three questions.

\begin{problem}\label{pb:1}
Does $M(t)^2/t$ (resp.\ $M(t)^2/\log t$) converge as $t\to\infty$ when $N=1$ (resp.\ $N=2$)? If so, what are the limits? In particular, are they determined by the zeroth moment $P=\int_{\R^N}u_1(x)\,\dd x$ alone, or do they depend on further properties of the initial data, such as the quantity $I_{0,N}$?
\end{problem}

\begin{problem}\label{pb:2}
What is the precise criterion separating growth from boundedness of $M(t)$? For $N\leq2$, the condition $P\neq0$ is sufficient for the unbounded growth of $M(t)$. It cannot, however, be the fundamental obstruction, since $P$ is not defined for a general initial velocity $u_1\in L^2(\R^N)$, whereas $M(t)$ is well defined for every finite-energy solution. Is there a condition involving $u_1$ only, formulated in a space that behaves naturally under translations and dilations, which is both necessary and sufficient for the boundedness of $M(t)$?
\end{problem}

\begin{problem}\label{pb:3}
Is infinite-time $L^2$-blowup a genuinely low-dimensional phenomenon? Or is it merely the trace, in the particular scale $L^2=\dH^0$, of a mechanism that is present in every dimension and becomes visible once the right space is chosen?
\end{problem}

The answers to these questions turn out to be remarkably clean. Moreover, all three are governed by the same low-frequency mechanism, which we now describe.

\subsection{The mechanism}\label{ss:mechanism}

Taking the spatial Fourier transform of \eqref{eq:wave}--\eqref{eq:data}, we find
\begin{equation}\label{eq:fourier-solution-intro}
\hat u(t,\xi)=\frac{\sin(t\abs{\xi})}{\abs{\xi}}\,\hat u_1(\xi)
+\cos(t\abs{\xi})\,\hat u_0(\xi).
\end{equation}

The two terms in \eqref{eq:fourier-solution-intro} play very different roles. Since $\abs{\cos(t\abs{\xi})}\leq1$, Plancherel's theorem gives $\norm{\cos(t\abs{D})u_0} \leq \norm{u_0}$. Thus the contribution of the initial position $u_0$ remains uniformly bounded in $L^2$, and it cannot by itself produce any growth of $M(t)$. In particular, it is negligible at the leading orders $\sqrt{t}$ in dimension one and $\sqrt{\log t}$ in dimension two.

The situation is different for the term generated by the initial velocity, whose Fourier multiplier is $\sin(t\abs{\xi})/\abs{\xi}$. This multiplier is bounded for each fixed $t$, but its bound grows with $t$, since
$$
\left| \frac{\sin(t\abs{\xi})}{\abs{\xi}} \right| \leq t.
$$
Moreover, near $\xi=0$ we have
$$
\frac{\sin(t\abs{\xi})}{\abs{\xi}} \sim t \qquad\text{when }t\abs{\xi}\ll1.
$$
It is therefore the low-frequency part of the initial velocity that decides whether $M(t)$ remains bounded or grows.

Assume for the moment that $u_1\in L^1(\R^N)$. Then $\hat u_1$ is continuous and, with the Fourier convention used here, $\hat u_1(0) = \int_{\R^N}u_1(x)\,\dd x = P$. Consequently, $\hat u_1(\xi)=P+o(1)$ as $\xi\to0$. If $P\neq0$, it follows that, near the origin,
$$
\frac{\abs{\hat u_1(\xi)}^2}{\abs{\xi}^2} \sim \frac{\abs{P}^2}{\abs{\xi}^2}.
$$
Passing to polar coordinates, the corresponding singular integral behaves like $\int_0^1 r^{N-3}\,\dd r$, which is finite if and only if $N\geq3$. Thus the dimensional restriction in \eqref{eq:ike1}--\eqref{eq:ike2} comes entirely from the interaction between the singular factor $\abs{\xi}^{-2}$ and the $N$-dimensional volume element $r^{N-1}\,\dd r$.

More precisely, the leading part of the squared $L^2$-norm generated by $u_1$ is governed by
$$
\abs{P}^2 \int_{\abs{\xi}\leq1} \frac{\sin^2(t\abs{\xi})}{\abs{\xi}^2}\,\dd\xi.
$$
In dimension $N=1$, up to the factor $2$ coming from the two directions of the real line, this reduces to $\int_0^1\sin^2(tr)\,r^{-2}\,\dd r$. The change of variables $s=tr$ gives
$$
\int_0^1\frac{\sin^2(tr)}{r^2}\,\dd r = t\int_0^t\frac{\sin^2s}{s^2}\,\dd s.
$$
Since $\int_0^\infty s^{-2}\sin^2s\,\dd s<\infty$, the integral is asymptotic to a constant multiple of $t$, and we obtain $M(t)^2\sim {c_1}\abs{P}^2t$. In this case, the rescaling $\xi=\eta/t$ turns the leading expression into a fixed integral, so that the dominant frequencies are of order $t^{-1}$.

In dimension $N=2$, polar coordinates lead instead to $\int_0^1\sin^2(tr)\,r^{-1}\,\dd r$. After the same change of variables, we get
$$
\int_0^1\frac{\sin^2(tr)}{r}\,\dd r = \int_0^t\frac{\sin^2s}{s}\,\dd s,
$$
and the last integral grows like $\tfrac12\log t$. Unlike the one-dimensional case, the logarithmic growth is not concentrated at the frequency scale $r\sim t^{-1}$. It accumulates over the whole intermediate region $t^{-1}\lesssim r\lesssim1$. Indeed, on this region the factor $\sin^2(tr)$ oscillates with average $\tfrac12$, while the measure $\dd r/r$ gives the same weight to each dyadic frequency scale. This produces $M(t)^2\sim {c_2}\abs{P}^2\log t$. The exact constants {$c_1$ and $c_2$} depend on the normalization of the Fourier transform, and they are identified in Section~\ref{sec:profiles}.

It is instructive to compare this Fourier-space mechanism with the physical-space picture. In dimension one, the d'Alembert formula {\eqref{eq:dalembert}} shows that a localized initial velocity with total mass $P$ produces an expanding plateau of height $P/2$. At leading order, $u(t,x)\approx\frac{P}{2}\mathbf 1_{\{\abs{x}<t\}}$, and consequently
\begin{equation}\label{eq:heur1}
M(t)^2 \approx \frac{P^2}{4}\cdot 2t = \frac{P^2}{2}\,t.
\end{equation}
Thus the one-dimensional growth is produced by a profile of essentially constant height whose support expands linearly in time.

In dimension two, {Poisson's formula \cite[Section~2.4]{Evans2010}} shows that the solution has, for $\abs{x}<t$, the tail $u(t,x)\approx \tfrac{P}{2\pi}(t^2-\abs{x}^2)^{-1/2}$. We infer that
\begin{equation}\label{eq:heur2}
M(t)^2\approx\int_{\abs{x}<t-\delta}\frac{P^2}{4\pi^2}
\frac{\dd x}{t^2-\abs{x}^2}
=\frac{P^2}{4\pi}\log\frac{t}{2\delta}+O(1)
\approx\frac{P^2}{4\pi}\log t .
\end{equation}
Both heuristics are confirmed exactly by Corollary~\ref{cor:L2} below. A third, purely algebraic, consistency check is available. {It relies on an averaging device that is classical in the study of finite-time blowup for semilinear wave equations $u_{tt}-\Delta u=\abs{u}^p$, and goes back to Kato \cite{Kato1980}, Glassey \cite{Glassey1981} and Sideris \cite{Sideris1984}. In that setting, one considers the spatial mean $\int_{\R^N}u(t,x)\dd x$, whose second time derivative equals $\int_{\R^N}\abs{u}^p\dd x$, and one bounds the latter from below by H\"older's inequality on the support of $u(t,\cdot)$, which is controlled by finite speed of propagation. For the free equation, the second derivative of the spatial mean vanishes.} Integrating \eqref{eq:wave} formally in $x$, we get $\frac{d}{dt}\int u\dd x=\int u_t\dd x$ and $\frac{d^2}{dt^2}\int u\dd x=\int\Delta u\dd x=0$, so that
\begin{equation}\label{eq:massconservation}
\int_{\R^N}u(t,x)\dd x=Q+Pt,\qquad Q:=\int_{\R^N}u_0(x)\dd x
\end{equation}
{(this is rigorous, for instance, for compactly supported data, by finite speed of propagation)}. In dimension one, {if the data are supported in $[-a,a]$, then $u(t,\cdot)$ is supported in $[-a-t,a+t]$, and the Cauchy--Schwarz inequality yields $\abs{Q+Pt}\le\sqrt{2(t+a)}\,M(t)$. It follows that} $M(t)^2\ge\tfrac{P^2}{2}t(1+o(1))$. Thus the constant $P^2/2$ in \eqref{eq:heur1} is exactly the Cauchy--Schwarz value. {Moreover, equality in the Cauchy--Schwarz inequality holds only for functions that are constant on their support, and the asymptotic profile is precisely such a flat plateau.} The growth is a spreading effect, not a concentration effect.

In dimension $N$, the same argument, with $u(t,\cdot)$ supported in the ball $\{\abs{x}\le t+a\}$, gives $\abs{Q+Pt}\le\abs{B_1}^{1/2}(t+a)^{N/2}M(t)$, where $\abs{B_1}$ denotes the volume of the unit ball, and hence
$$
M(t)\ \ge\ \frac{\abs{P}}{\abs{B_1}^{1/2}}\,t^{1-N/2}\bigl(1+o(1)\bigr)\qquad\text{as }t\to\infty .
$$
For $N=2$, this yields only $\liminf_{t\to\infty}M(t)\ge\abs{P}/\sqrt\pi$, and not the logarithmic growth. In two dimensions, the solution is far from constant on its support, and the growth is produced by the tail in \eqref{eq:heur2} rather than by a plateau, so that the Cauchy--Schwarz inequality is very far from equality. {We are not aware of an earlier explicit use of \eqref{eq:massconservation} to bound the $L^2$-norm of free waves from below.} In any case, this argument applies only to compactly supported data, provides only a lower bound, and does not capture the two-dimensional growth. It is therefore no substitute for the analysis of \cite{Ikehata2023} or for the Fourier analysis of the present paper, and we use it only as a consistency check.

Finally, it is worth contrasting this with the heat semigroup. There, the same low-frequency analysis of $\hat u_1(0)=P$ gives $\norm{e^{t\Delta}u_1}\sim c_N\abs{P}\,t^{-N/4}$, that is, {decay}. The difference lies entirely in the multiplier. The heat multiplier $e^{-t\abs{\xi}^2}$ is bounded by $1$ and localizes the solution at low frequencies, {whereas the size of $\sin(t\abs{\xi})/\abs{\xi}$ near $\xi=0$ grows like $t$}. In this sense, the wave equation amplifies exactly what the heat equation dissipates.

\subsection{Main results}\label{ss:results}

\medskip
\noindent\textbf{Notation.} We normalize the Fourier transform by
$$
\hat f(\xi) = \mathcal F_{x\to\xi}(f)(\xi) := \int_{\R^N}e^{-ix\cdot\xi}f(x)\,\dd x,
$$
so that Plancherel's identity reads
$$
\norm{f}_{L^2(\R^N)}^2 = \frac{1}{(2\pi)^N} \norm{\hat f}_{L^2(\R^N)}^2.
$$
For $s\in\R$, we define
$$
\norm{f}_{\dH^{-s}(\R^N)}^2 := \frac{1}{(2\pi)^N} \int_{\R^N} \abs{\xi}^{-2s}\abs{\hat f(\xi)}^2\,\dd\xi = \norm{\abs{\nabla}^{-s}f}_{L^2(\R^N)}^2 \in[0,+\infty],
$$
and we write $f\in\dH^{-s}(\R^N)$ when this quantity is finite {(see \cite[Chapter~1]{BCD2011} for homogeneous Sobolev spaces)}. We also set
$$
\sigma_{N-1} := \abs{\Sph^{N-1}} = \frac{2\pi^{N/2}}{\Gamma(N/2)},
$$
with the convention $\sigma_0=2$. Finally, $\gamma$ denotes the Euler--Mascheroni constant, and $C>0$ denotes a harmless constant whose value may change from line to line.
{Several results below only require $u_0\in L^2(\R^N)$. For $[u_0,u_1]\in L^2(\R^N)\times L^2(\R^N)$, the solution of \eqref{eq:wave}--\eqref{eq:data} is understood as the function defined by the Fourier representation \eqref{eq:fourier-solution} below. It belongs to $C([0,\infty);L^2(\R^N))$ and coincides with the weak solution described above whenever $[u_0,u_1]\in H^1(\R^N)\times L^2(\R^N)$.}
\medskip

Our first result gives a complete answer to Problem~\ref{pb:2}. It requires no spatial integrability assumption beyond the energy space and, in particular, makes no reference to the zeroth moment $P$.

\begin{theorem}[Boundedness dichotomy]\label{th:dichotomy}
Let $N\geq1$, let $[u_0,u_1]\in H^1(\R^N)\times L^2(\R^N)$, and let $u$ be the solution of \eqref{eq:wave}--\eqref{eq:data}. Then
$$
\sup_{t\geq0}\norm{u(t,\cdot)}<\infty \qquad\Longleftrightarrow\qquad u_1\in\dH^{-1}(\R^N).
$$
If $u_1\in\dH^{-1}(\R^N)$, then
\begin{equation}\label{eq:equipartition}
\lim_{t\to\infty}\norm{u(t,\cdot)}^2 = \frac12 \left( \norm{u_1}_{\dH^{-1}}^2+\norm{u_0}^2 \right).
\end{equation}
If $u_1\notin\dH^{-1}(\R^N)$, then
$$
\norm{u(t,\cdot)}\longrightarrow+\infty \qquad\text{as }t\to\infty.
$$
\end{theorem}

Theorem~\ref{th:dichotomy} shows that the boundedness of $M(t)$ is governed neither directly by the space dimension nor by a moment condition. The exact criterion is $u_1\in\dH^{-1}(\R^N)$, which involves only the initial velocity. This condition is invariant under spatial translations and behaves naturally under scaling. Indeed, if $f_\lambda(x):=f(\lambda x)$, then
$$
\norm{f_\lambda}_{\dH^{-1}} = \lambda^{-N/2-1}\norm{f}_{\dH^{-1}}.
$$

The connection with the zeroth moment appears under the additional assumption $u_1\in L^1(\R^N)$. In that case, $\hat u_1$ is continuous and $\hat u_1(0) = P$. If $P\neq0$, then $\abs{\hat u_1(\xi)}$ is bounded away from zero in a neighborhood of $\xi=0$. Consequently, the low-frequency part of the $\dH^{-1}$-norm behaves like $\int_{\abs{\xi}\leq1}\abs{\xi}^{-2}\,\dd\xi=\sigma_{N-1}\int_0^1r^{N-3}\,\dd r$, which diverges precisely when $N\leq2$. We conclude that, for $u_1\in L^1(\R^N)\cap L^2(\R^N)$ with $P\neq0$,
$$
u_1\notin\dH^{-1}(\R^N) \qquad\Longleftrightarrow\qquad N\leq2.
$$
This recovers and explains the dimensional restriction in \cite{Ikehata2023}.

Formula \eqref{eq:equipartition} is an asymptotic averaging identity. By the Fourier representation \eqref{eq:fourier-solution-intro}, when $u_1\in\dH^{-1}(\R^N)$ the squared amplitudes of both terms are integrable in $\xi$. The oscillatory factors satisfy
$$
\sin^2(t\abs{\xi}) = \frac12-\frac12\cos(2t\abs{\xi}),
\qquad
\cos^2(t\abs{\xi}) = \frac12+\frac12\cos(2t\abs{\xi}).
$$
After integration in $\xi$, the oscillatory parts vanish as $t\to\infty$, and each squared component contributes one half of its non-oscillatory mass. The mixed term vanishes for the same reason, and we obtain \eqref{eq:equipartition}.

In other words, \eqref{eq:equipartition} is an equipartition statement one derivative below the energy level. A direct computation in Fourier variables shows that, when $u_1\in\dH^{-1}(\R^N)$, the quantity $\norm{u(t,\cdot)}^2+\norm{u_t(t,\cdot)}_{\dH^{-1}}^2$ is conserved and equals $\norm{u_0}^2+\norm{u_1}_{\dH^{-1}}^2$. Formula \eqref{eq:equipartition} says that, asymptotically, this conserved quantity is shared equally between its two parts, in the spirit of the classical equipartition of energy for the wave equation \cite{Brodsky1967,Goldstein1969,Duffin1970}.

Our second and main result answers Problems~\ref{pb:1} and~\ref{pb:3}. For $0\le s<N/2$, we set
\begin{equation}\label{eq:theta}
\theta=\theta(N,s):=2+2s-N ,
\end{equation}
and, when $\theta>0$ (equivalently $\tfrac{N-2}{2}<s<\tfrac N2$, so that $\mu:=3+2s-N\in(1,3)$),
\begin{equation}\label{eq:kappa}
\kappa_{N,s}:=\int_{\R^N}\frac{\sin^2\abs{\eta}}{\abs{\eta}^{2+2s}}\dd\eta
=\frac{2^{\mu-3}\,\pi\,\sigma_{N-1}}{\Gamma(\mu)\,\abs{\cos(\pi\mu/2)}}
\in(0,\infty).
\end{equation}

\begin{theorem}[Exact first-order profiles]\label{th:profiles}
Let $N\ge1$ and $0\le s<N/2$, and let $u$ solve \eqref{eq:wave}--\eqref{eq:data} with
\[
u_1\in L^1(\R^N)\cap L^2(\R^N),\qquad u_0\in L^2(\R^N)\cap\dH^{-s}(\R^N).
\]
Put $P=\int_{\R^N}u_1\dd x$ and $\theta=2+2s-N$. Then the following statements hold.
\begin{enumerate}
\item[(i)] {(Polynomial regime, $\theta>0$.)}
\[
\lim_{t\to\infty}\ t^{-\theta}\,\norm{u(t,\cdot)}_{\dH^{-s}}^2
=\frac{\kappa_{N,s}}{(2\pi)^N}\,P^2 .
\]
\item[(ii)] {(Logarithmic regime, $\theta=0$.)}
\[
\lim_{t\to\infty}\ (\log t)^{-1}\,\norm{u(t,\cdot)}_{\dH^{-s}}^2
=\frac{\sigma_{N-1}}{2\,(2\pi)^N}\,P^2 .
\]
\item[(iii)] {(Convergent regime, $\theta<0$.)} {The velocity $u_1$ automatically belongs to $\dH^{-s-1}(\R^N)$}, and
\[
\lim_{t\to\infty}\norm{u(t,\cdot)}_{\dH^{-s}}^2
=\frac12\Bigl(\norm{u_1}_{\dH^{-s-1}}^2+\norm{u_0}_{\dH^{-s}}^2\Bigr).
\]
\end{enumerate}
\end{theorem}

Several comments are in order.

\smallskip
\noindent{(a) Only $P$ survives at leading order.} In the polynomial and logarithmic regimes, the leading asymptotic coefficient depends on the data only through the single number $P=\int_{\R^N}u_1(x)\,\dd x$. The initial position $u_0$ is negligible at the relevant growing scale, and so is every feature of $u_1$ other than its total mass. More precisely, if $\theta>0$, then
$$
\norm{u(t,\cdot)}_{\dH^{-s}}^2 = \frac{\kappa_{N,s}}{(2\pi)^N}P^2t^\theta + o(t^\theta),
$$
whereas, in the critical case $\theta=0$,
$$
\norm{u(t,\cdot)}_{\dH^{-s}}^2 = \frac{\sigma_{N-1}}{2(2\pi)^N}P^2\log t + o(\log t).
$$
Thus, whenever $P\neq0$, the lower bounds in \eqref{eq:ike1}--\eqref{eq:ike2} identify the correct quantity, namely $\abs{P}$, and Corollary~\ref{cor:L2} below gives the exact value of the leading coefficient. The quantity $I_{0,N}$ appearing in the upper estimates affects at most lower-order terms and does not enter the principal term. Hence the answer to Problem~\ref{pb:1} is affirmative. When $P\neq0$, the exact first-order asymptotics is determined by $\abs{P}$ alone.

\medskip
\noindent{(b) No first moment is needed at leading order.} The weighted assumption $u_1\in L^{1,1}(\R^N)$ used in \cite{Ikehata2023} {for the lower bounds} is not needed to identify the leading-order growth. The proof of Theorem~\ref{th:profiles} uses only the continuity of $\hat u_1$ at the origin, namely
$$
\hat u_1(\xi)\longrightarrow\hat u_1(0)=P \qquad\text{as }\xi\to0,
$$
which follows from $u_1\in L^1(\R^N)$. No quantitative rate for this convergence is required. Example~\ref{ex:cauchy} below provides an initial velocity with $u_1\in L^1(\R)\setminus L^{1,1}(\R)$ to which Theorem~\ref{th:profiles} applies. In this sense, the theorem extends the class of data covered by \eqref{eq:ike1}--\eqref{eq:ike2}. The distinction between an exact leading term and a bounded remainder is important here. Under the sole assumption $u_1\in L^1\cap L^2$, the critical asymptotics reads
$$
\norm{u(t,\cdot)}_{\dH^{-s}}^2 = \frac{\sigma_{N-1}}{2(2\pi)^N}P^2\log t + o(\log t).
$$
A sharper remainder, such as $O(1)$, generally requires quantitative information on the behavior of $\hat u_1$ near the origin, and hence stronger assumptions on the data.

\medskip
\noindent{(c) A critical regularity.} The trichotomy is governed by the sign of $\theta=2+2s-N$, that is, by the position of $s$ relative to the critical index
\begin{equation}\label{eq:scrit}
s_c:=\frac{N-2}{2}.
\end{equation}
Indeed,
$$
\theta>0 \iff s>s_c, \qquad \theta=0 \iff s=s_c, \qquad \theta<0 \iff s<s_c.
$$
The meaning of $s_c$ is transparent in Fourier space. The propagator of the initial velocity contributes the factor $\sin(t\abs{\xi})/\abs{\xi}$, which introduces one inverse derivative at low frequency. Measuring the solution in $\dH^{-s}$ introduces the additional weight $\abs{\xi}^{-s}$. Thus the squared low-frequency amplitude carries the weight $\abs{\xi}^{-2-2s}$. Combining this singularity with the radial volume element $\abs{\xi}^{N-1}\,\dd\abs{\xi}$, we arrive at the model integral
$$
\int_0^1 r^{N-3-2s}\,\dd r .
$$
This integral diverges like a power when $s>s_c$, diverges logarithmically when $s=s_c$, and converges when $s<s_c$. The three asymptotic regimes are the dynamical counterparts of these three alternatives. {Note also that $\dH^{-s_c}(\R^N)$ is the space in which the norm of the solution is invariant under the scaling $u(t,x)\mapsto\lambda^{N-1}u(\lambda t,\lambda x)$, which preserves $P$. This is consistent with the fact that the growth at the critical index is only logarithmic.}

For $N=2$, we have $s_c=0$, so that the critical space is $\dH^{-s_c}(\R^2)=L^2(\R^2)$. The logarithmic regime therefore recovers and sharpens the two-dimensional $L^2$-result of \cite{Ikehata2023}. For $N=1$, we have $s_c=-\frac12$. Since Theorem~\ref{th:profiles} is stated for $s\geq0$, every admissible space $\dH^{-s}$ lies in the polynomial regime. In particular, $L^2(\R)=\dH^0(\R)$ lies strictly above the critical index and exhibits polynomial growth. Formally, the critical space corresponding to $s=s_c=-1/2$ is $\dH^{1/2}(\R)$, which is not part of the negative-order scale considered in Theorem~\ref{th:profiles}. For $N\geq3$, we have $s_c\geq1/2$. Hence $L^2=\dH^0$ lies in the convergent regime, while the negative-order space $\dH^{-(N-2)/2}(\R^N)$ is critical and exhibits logarithmic growth when $P\neq0$. Thus, in every dimension $N\geq2$, logarithmic growth occurs at the critical regularity $s=s_c$.

This answers Problem~\ref{pb:3}. The growth mechanism is not peculiar to low dimensions. It is a low-frequency phenomenon tied to a critical regularity, and the dimensions $N=1$ and $N=2$ are distinguished only because the unweighted space $L^2=\dH^0$ lies in the polynomial and logarithmic regimes, respectively. The three regimes are summarized in Table~\ref{tab:regimes}.

\begin{table}[ht]
\centering
\renewcommand{\arraystretch}{1.35}
\begin{tabular}{ >{\centering\arraybackslash}m{3.1cm} >{\centering\arraybackslash}m{3.7cm} >{\centering\arraybackslash}m{5.8cm}}
\hline
\textbf{Regime} & \textbf{Condition} & $\norm{u(t,\cdot)}_{\dH^{-s}}^2$ as $t\to\infty$ \\
\hline
Polynomial & $0\leq s<\dfrac N2$ and $s>s_c$ & $\dfrac{\kappa_{N,s}}{(2\pi)^N}P^2 t^{2s+2-N} +o\bigl(t^{2s+2-N}\bigr)$ \\[2mm]
Logarithmic & $0\leq s<\dfrac N2$ and $s=s_c$ & $\dfrac{\sigma_{N-1}}{2(2\pi)^N}P^2\log t +o(\log t)$ \\[2mm]
Convergent & $0\leq s<\dfrac N2$ and $s<s_c$ & $\dfrac12\left( \norm{u_1}_{\dH^{-s-1}}^2 + \norm{u_0}_{\dH^{-s}}^2 \right) +o(1)$ \\
\hline
\end{tabular}
\medskip
\caption{The three regimes of Theorem~\ref{th:profiles} for $u_1\in L^1(\R^N)\cap L^2(\R^N)$ with $P=\int_{\R^N}u_1(x)\,\dd x\neq0$. For $s=0$, the $L^2$-norm lies in the polynomial regime when $N=1$, in the logarithmic regime when $N=2$, and in the convergent regime when $N\geq3$.}
\label{tab:regimes}
\end{table}

Specializing Theorem~\ref{th:profiles} to $s=0$, we obtain the exact form of the low-dimensional growth estimates of \cite{Ikehata2023}.

\begin{corollary}\label{cor:L2}
Let $N\geq1$, let $u_0\in L^2(\R^N)$ and $u_1\in L^1(\R^N)\cap L^2(\R^N)$, and set $P:=\int_{\R^N}u_1(x)\,\dd x$. Then the following statements hold.
\begin{enumerate}
\item[(i)] If $N=1$, then
$$
\lim_{t\to\infty} \frac{\norm{u(t,\cdot)}^2}{t} = \frac{P^2}{2}.
$$
\item[(ii)] If $N=2$, then
$$
\lim_{t\to\infty} \frac{\norm{u(t,\cdot)}^2}{\log t} = \frac{P^2}{4\pi}.
$$
\item[(iii)] If $N\geq3$, then $u_1\in\dH^{-1}(\R^N)$ automatically, and
$$
\lim_{t\to\infty} \norm{u(t,\cdot)}^2 = \frac12 \left( \norm{u_1}_{\dH^{-1}}^2+\norm{u_0}^2 \right).
$$
\end{enumerate}
\end{corollary}

The constants in the one- and two-dimensional limits agree exactly with those predicted by the physical-space heuristics \eqref{eq:heur1} and \eqref{eq:heur2}. This is a useful consistency check, and it confirms that the Fourier analysis captures the whole leading-order contribution.

Our next two results refine the leading-order profiles by identifying the second-order terms in the two regimes where the leading term is unbounded. We begin with the one-dimensional polynomial regime.

\begin{theorem}[Second-order expansion, $N=1$]\label{th:secondorder1}
Let $N=1$, $u_1\in {L^{1,1}(\R)}\cap L^2(\R)$ and $u_0\in L^{1,1}(\R)\cap L^2(\R)$. Set $P=\int_\R u_1\dd x$ and $Q=\int_\R u_0\dd x$. Then, as $t\to\infty$,
\begin{equation}\label{eq:secondorder1}
\norm{u(t,\cdot)}^2=\frac{P^2}{2}\,t+\mathcal C_1+o(1),
\qquad
\mathcal C_1:=\frac{1}{4\pi}\int_{\R}
\frac{\abs{\hat u_1(\xi)}^2-P^2}{\xi^2}\dd\xi
+\frac{\norm{u_0}^2}{2}+\frac{PQ}{2},
\end{equation}
the integral being absolutely convergent.
\end{theorem}

{The moment assumption $u_1\in L^{1,1}(\R)$ is sharp on the scale $L^{1,\alpha}(\R)$: by Example~\ref{ex:cauchy}, the conclusion fails for a velocity that belongs to $L^{1,\alpha}(\R)$ for every $\alpha<1$; see also Remark~\ref{rem:moment-sharp}.}

The second expansion concerns the critical regularity. It holds in every dimension $N\ge2$ and contains, in particular, the second-order expansion of $\norm{u(t,\cdot)}_{L^2(\R^2)}$.

\begin{theorem}[Second-order expansion at critical regularity]\label{th:secondorder2}
Let $N\ge2$, $s_c=(N-2)/2$, $u_1\in L^{1,1}(\R^N)\cap L^2(\R^N)$ and $u_0\in L^1(\R^N)\cap L^2(\R^N)$. Let
\begin{equation}\label{eq:g}
g(r):=\int_{\Sph^{N-1}}\bigl\lvert\hat u_1(r\omega)\bigr\rvert^2\dd\omega,
\qquad r>0,\qquad\text{so that}\quad g(0^+)=\sigma_{N-1}P^2 .
\end{equation}
Then, as $t\to\infty$,
\begin{equation}\label{eq:secondorder2}
\norm{u(t,\cdot)}^2_{\dH^{-s_c}}
=\frac{\sigma_{N-1}P^2}{2\,(2\pi)^N}\,\log t+\mathcal C_2+o(1),
\end{equation}
where
\begin{equation}\label{eq:C2}
\mathcal C_2:=\frac{1}{2\,(2\pi)^N}\left[\sigma_{N-1}P^2\log2
+\int_0^\infty\frac{g(r)-\sigma_{N-1}P^2e^{-r}}{r}\dd r\right]
+\frac{\norm{u_0}^2_{\dH^{-s_c}}}{2},
\end{equation}
the integral being absolutely convergent. In particular, for $N=2$,
\begin{equation}\label{eq:secondorder2d}
\norm{u(t,\cdot)}^2_{L^2(\R^2)}
=\frac{P^2}{4\pi}\log t
+\frac{1}{8\pi^2}\left[2\pi P^2\log2
+\int_0^\infty\frac{g(r)-2\pi P^2e^{-r}}{r}\dd r\right]
+\frac{\norm{u_0}^2}{2}+o(1).
\end{equation}
\end{theorem}

The auxiliary exponential $e^{-r}$ in \eqref{eq:C2} is of course arbitrary. Any fixed function $\varphi$ with $\varphi(0)=1$, $\varphi$ Lipschitz near $0$ and $\varphi(r)/r$ integrable at infinity may be used instead, provided the term $\sigma_{N-1}P^2\log2$ is modified accordingly; see Remark~\ref{rem:normalization}. We have chosen $e^{-r}$ because it makes the examples of Section~\ref{sec:examples} computable in closed form.

Finally, we return to the original motivation of \cite{Ikehata2023} and use the sharp asymptotics of the $L^2$-norm obtained above to improve the local energy decay estimate in two dimensions. For $R>0$, we set $B_R:=\left\{x\in\R^2:\abs{x}\leq R\right\}$ and define
\begin{equation}\label{eq:localenergy}
E_R(t) := {\frac12}\int_{B_R} \left( \abs{u_t(t,x)}^2+\abs{\nabla u(t,x)}^2 \right)\,\dd x .
\end{equation}

\begin{theorem}[Local energy decay with explicit constant, $N=2$]\label{th:LED}
Let $N=2$, $R>0$, and let $[u_0,u_1]\in H^1(\R^2)\times L^2(\R^2)$ satisfy {$u_1\in L^1(\R^2)$} and
$\int_{\R^2}\abs{x}\bigl(\abs{u_1(x)}^2+\abs{\nabla u_0(x)}^2\bigr)\dd x<\infty$.
Then
\begin{equation}\label{eq:LEDsharp}
\limsup_{t\to\infty}\ \frac{t\,E_R(t)}{\sqrt{\log t}}\ \le\
\frac{\abs{P}}{2}\sqrt{\frac{E(0)}{2\pi}},
\end{equation}
and the right-hand side does not depend on $R$. If, moreover, $u_1\in\dH^{-1}(\R^2)$ (for instance, if $u_1\in L^{1,1}(\R^2)$ and $P=0$), then
\begin{equation}\label{eq:LEDbounded}
\limsup_{t\to\infty}\ t\,E_R(t)\ \le\ K_0
+\frac{\sqrt{E(0)}}{2}\Bigl(\norm{u_1}_{\dH^{-1}}^2+\norm{u_0}^2\Bigr)^{1/2},
\end{equation}
where $K_0$ is the constant defined in \eqref{eq:K0} below.
\end{theorem}

{Whether the rate $t^{-1}\sqrt{\log t}$ in \eqref{eq:LEDsharp} is optimal for the local energy itself is, as far as we know, an open problem; see Problem~\ref{pb:LEDsharp}.}

\subsection{Comments on the literature}\label{ss:lit}
Part~(iii) of Theorem~\ref{th:profiles}, and hence the implication
$$
u_1\in\dH^{-1}(\R^N) \quad\Longrightarrow\quad \sup_{t\geq0}\norm{u(t,\cdot)}<\infty
$$
in Theorem~\ref{th:dichotomy}, follows from a direct Fourier argument based on the Riemann--Lebesgue lemma. This argument is closely related in spirit to classical equipartition results for the free wave equation {\cite{Brodsky1967,Goldstein1969,Duffin1970}}. We nevertheless include the proof for completeness and, more importantly, because the same Fourier identity also reveals the polynomial and logarithmic regimes when it is examined at and above the critical regularity. From this perspective, the bounded and unbounded regimes are not separate phenomena. They arise from the same oscillatory integral, and the transition between them is decided by the integrability of its low-frequency amplitude. Below the critical index, the amplitude is integrable and the oscillations produce a finite averaged limit. At the critical index, integrability fails logarithmically, which produces logarithmic growth. Above the critical index, the low-frequency singularity produces polynomial growth. In each case, the leading constant can be computed explicitly.

Ikehata's results have recently been extended in several directions. Chen and Ikehata \cite{ChenIkehata2025} obtained optimal two-sided bounds, with unspecified constants, for the $L^2$-norm of the solution when the initial velocity has fractional derivatives $\abs{\nabla}^{s}u_1$ in $L^1\cap L^2$, in every dimension. Since $\norm{u(t,\cdot)}_{\dH^{-s}}$ is the $L^2$-norm of the solution with data $\abs{\nabla}^{-s}u_0$ and $\abs{\nabla}^{-s}u_1$, their polynomial, logarithmic and bounded regimes correspond, up to multiplicative constants and under weighted $L^1$ assumptions for the lower bounds, to the three regimes of Theorem~\ref{th:profiles}. They also raised the question of the minimal regularity of the data that keeps $\norm{u(t,\cdot)}$ bounded, and Theorem~\ref{th:dichotomy} answers this question.\ifmentionchen{}  Infinite-time $L^2$-blowup has also been established for the plate equation \cite{Ikehata2024plate}, for wave equations with higher-order derivative terms \cite{IkehataLi2023}, for the fractional wave equation \cite{IkedaZhao2024}, and for wave equations with logarithmic damping \cite{CharaoIkehata2022log}.

The principal contributions of the present paper are the exact leading constants in parts~(i) and~(ii) of Theorem~\ref{th:profiles}, {obtained under the sole assumption $u_1\in L^1\cap L^2$}, the necessary and sufficient boundedness criterion of Theorem~\ref{th:dichotomy}, the second-order expansions \eqref{eq:secondorder1} and \eqref{eq:secondorder2}, and the local energy estimate \eqref{eq:LEDsharp} with an explicit asymptotic constant. We are not aware of previous results establishing these conclusions under the present assumptions. {By contrast, the elementary physical-space arguments of Subsections~\ref{ss:ikehata} and~\ref{ss:mechanism} only give alternative proofs of special cases of \cite{Ikehata2023} for compactly supported data. The phenomenon of infinite-time $L^2$-blowup for free waves, and the motivation for studying it, are due to \cite{Ikehata2023}; together with \cite{ChenIkehata2025}, that work is the basis of the present paper.}

The remainder of the paper is organized as follows. Section~\ref{sec:prelim} collects the elementary harmonic-analysis lemmas used throughout the paper, all of which are variants of the Riemann--Lebesgue lemma. Section~\ref{sec:dichotomy} is devoted to the proof of Theorem~\ref{th:dichotomy}, while Section~\ref{sec:profiles} establishes Theorem~\ref{th:profiles} and Corollary~\ref{cor:L2}. Section~\ref{sec:secondorder} contains the two second-order asymptotic expansions stated in Theorems~\ref{th:secondorder1} and~\ref{th:secondorder2}. In Section~\ref{sec:examples}, we present six explicit examples, including an exactly solvable one-dimensional case and an example demonstrating the sharpness of the moment hypothesis in Theorem~\ref{th:secondorder1}. Section~\ref{sec:LED} presents the application to local energy decay, and Section~\ref{sec:open} concludes with further remarks and open problems.
\section{Preliminaries and supporting results}\label{sec:prelim}

In this section, we collect the Fourier identities and oscillatory integral estimates used throughout the paper. We begin with the Fourier representation of solutions to the free wave equation, and we split the relevant homogeneous Sobolev norm into the contributions of the initial velocity, of the initial position, and of their interaction. We then record radial versions of the Riemann--Lebesgue and Dirichlet lemmas, together with first and second-order asymptotic formulas for the critical logarithmic integral.

Taking the Fourier transform of \eqref{eq:wave}--\eqref{eq:data} in the space variable, we find that $w(t,\xi):=\hat u(t,\xi)$ solves the ordinary differential equation
$$
w_{tt}(t,\xi)+\abs{\xi}^2w(t,\xi)=0,
$$
depending on the parameter $\xi$, with initial conditions $w(0,\xi)=\hat u_0(\xi)$ and $w_t(0,\xi)=\hat u_1(\xi)$. Consequently,
\begin{equation}\label{eq:fourier-solution}
\hat u(t,\xi) = \frac{\sin(t\abs{\xi})}{\abs{\xi}}\,\hat u_1(\xi) + \cos(t\abs{\xi})\,\hat u_0(\xi),
\end{equation}
where the quotient $\sin(t\abs{\xi})/\abs{\xi}$ is understood to be equal to $t$ at $\xi=0$. Whenever the weighted Fourier integrals below are finite, expanding the square in \eqref{eq:fourier-solution} gives
\begin{equation}\label{eq:expansion}
(2\pi)^N\norm{u(t,\cdot)}_{\dH^{-s}}^2 = \mathcal A_s(t)+\mathcal B_s(t)+\mathcal D_s(t),
\end{equation}
where
\begin{align}
\mathcal A_s(t) &:= \int_{\R^N} \frac{\sin^2(t\abs{\xi})}{\abs{\xi}^{2s+2}} \abs{\hat u_1(\xi)}^2\,\dd\xi, \label{eq:A}\\
\mathcal B_s(t) &:= \int_{\R^N} \frac{\cos^2(t\abs{\xi})}{\abs{\xi}^{2s}} \abs{\hat u_0(\xi)}^2\,\dd\xi, \label{eq:B}\\
\mathcal D_s(t) &:= \int_{\R^N} \frac{\sin(2t\abs{\xi})}{\abs{\xi}^{2s+1}} \Real\left( \hat u_1(\xi)\overline{\hat u_0(\xi)} \right)\,\dd\xi. \label{eq:D}
\end{align}
Here we used the identity $2\sin\alpha\cos\alpha=\sin(2\alpha)$. By the Cauchy--Schwarz inequality, we have
\begin{equation}\label{eq:cross-CS}
\abs{\mathcal D_s(t)} \leq 2\sqrt{\mathcal A_s(t)}\sqrt{\mathcal B_s(t)}.
\end{equation}
In particular, if $\mathcal B_s(t)$ remains bounded and $\mathcal A_s(t)\to\infty$, we obtain
$$
\frac{\abs{\mathcal D_s(t)}}{\mathcal A_s(t)} \leq 2\sqrt{ \frac{\mathcal B_s(t)}{\mathcal A_s(t)} } \longrightarrow0.
$$
In this situation, the mixed term is of strictly lower order than $\mathcal A_s(t)$, and the leading growth is generated entirely by the initial velocity.

We shall use the following elementary facts. We begin with a radial version of the Riemann--Lebesgue lemma.

\begin{lemma}\label{lem:RL}
Let $F\in L^1(\R^N)$. Then
$$
\int_{\R^N}F(\xi)e^{i\lambda\abs{\xi}}\,\dd\xi \longrightarrow0 \qquad\text{as }\abs{\lambda}\to\infty.
$$
\end{lemma}

\begin{proof}
By polar coordinates, we have
$$
\int_{\R^N}F(\xi)e^{i\lambda\abs{\xi}}\,\dd\xi = \int_0^\infty G(r)e^{i\lambda r}\,\dd r,
\qquad
G(r) := r^{N-1} \int_{\Sph^{N-1}}F(r\omega)\,\dd\omega.
$$
Moreover, the triangle inequality and the polar-coordinate formula give
$$
\int_0^\infty\abs{G(r)}\,\dd r \leq \int_0^\infty r^{N-1} \int_{\Sph^{N-1}} \abs{F(r\omega)}\,\dd\omega\,\dd r = \norm{F}_{L^1(\R^N)}.
$$
Hence $G\in L^1(0,\infty)$, and the conclusion follows from the classical Riemann--Lebesgue lemma on the half-line {(see, e.g., \cite[Chapter~I]{SteinWeiss1971})}.
\end{proof}

The next result is a Dirichlet-type lemma.

\begin{lemma}\label{lem:dirichlet}
Let $h:(0,\infty)\to\R$ be measurable and such that
\[
h(0^+):=\lim_{\xi\to0^+}h(\xi)\ \text{ exists},\qquad
\frac{h(\xi)-h(0^+)}{\xi}\in L^1(0,1),\qquad
\frac{h(\xi)}{\xi}\in L^1(1,\infty).
\]
Then
$\displaystyle\int_0^\infty\frac{\sin(\lambda\xi)}{\xi}h(\xi)\dd\xi
\longrightarrow\frac{\pi}{2}h(0^+)$ as $\lambda\to+\infty$.
\end{lemma}

\begin{proof}
We split the integral at $\xi=1$. Since $h(\xi)/\xi\in L^1(1,\infty)$, the Riemann--Lebesgue lemma gives
$$
\int_1^\infty \frac{\sin(\lambda\xi)}{\xi}h(\xi)\,\dd\xi \longrightarrow0.
$$
On $(0,1)$, we write
$$
\int_0^1 \frac{\sin(\lambda\xi)}{\xi}h(\xi)\,\dd\xi = h(0^+) \int_0^1\frac{\sin(\lambda\xi)}{\xi}\,\dd\xi+ \int_0^1 \sin(\lambda\xi) \frac{h(\xi)-h(0^+)}{\xi}\,\dd\xi.
$$
After the change of variables $s=\lambda\xi$, the first integral becomes
$$
\int_0^\lambda\frac{\sin s}{s}\,\dd s \longrightarrow\frac{\pi}{2}.
$$
The second integral tends to zero by the {one-dimensional} Riemann--Lebesgue lemma, since $(h(\xi)-h(0^+))/\xi\in L^1(0,1)$. Combining the three contributions, we obtain the result.
\end{proof}

We now record a first-order formula for the critical logarithmic integral.

\begin{lemma}[Logarithmic lemma, first order]\label{lem:log}
Let $g:(0,\infty)\to[0,\infty)$ be bounded {and measurable}, with $g(0^+):=\lim_{r\to0^+}g(r)$ existing and $\int_1^\infty g(r)r^{-1}\dd r<\infty$. Then
\[
\int_0^\infty\frac{\sin^2(tr)}{r}\,g(r)\dd r
=\frac{g(0^+)}{2}\,\log t+o(\log t),\qquad t\to\infty.
\]
\end{lemma}

\begin{proof}
The contribution of $(1,\infty)$ is bounded by $\int_1^\infty g(r)r^{-1}\dd r$, hence it is $O(1)$. On $(0,1)$, we substitute $s=tr$ and use $\sin^2 s=\tfrac12(1-\cos2s)$. For $t\ge1$, we get
\begin{equation}\label{eq:sinc-log}
\int_0^1\frac{\sin^2(tr)}{r}\dd r=\int_0^t\frac{\sin^2 s}{s}\dd s
=\frac12\log t+O(1),
\end{equation}
because $\int_0^1\frac{1-\cos 2s}{2s}\dd s<\infty$ and $\int_1^t\frac{\cos 2s}{2s}\dd s$ converges as $t\to\infty$. Let $\varepsilon>0$, and choose $\delta\in(0,1)$ such that $\abs{g(r)-g(0^+)}\le\varepsilon$ on $(0,\delta)$. Using \eqref{eq:sinc-log}, we obtain
\[
\Bigl\lvert\int_0^1\frac{\sin^2(tr)}{r}\bigl(g(r)-g(0^+)\bigr)\dd r\Bigr\rvert
\le\varepsilon\Bigl(\frac12\log t+C\Bigr)
+\bigl(\norm{g}_\infty+\abs{g(0^+)}\bigr)\log\frac1\delta .
\]
Dividing by $\log t$ and letting first $t\to\infty$ and then $\varepsilon\to0$, we obtain the claim.
\end{proof}

The next lemma refines Lemma~\ref{lem:log} to second order. It is the main oscillatory-integral input in the proof of Theorem~\ref{th:secondorder2}. Under a Dini-type condition at the origin, the remainder left after subtracting the leading logarithmic term converges, and its limit can be computed explicitly.

\begin{lemma}[Logarithmic lemma, second order]\label{lem:log2}
Let $g:(0,\infty)\to[0,\infty)$ be bounded and measurable, with $g(0^+)$ existing, $\int_1^\infty g(r)r^{-1}\dd r<\infty$, and
\begin{equation}\label{eq:dini}
\frac{g(r)-g(0^+)}{r}\in L^1(0,1).
\end{equation}
Then, as $t\to\infty$,
\begin{equation}\label{eq:log2}
\int_0^\infty\frac{\sin^2(tr)}{r}\,g(r)\dd r
=\frac{g(0^+)}{2}\log t+\frac{g(0^+)}{2}\log 2
+\frac12\int_0^\infty\frac{g(r)-g(0^+)e^{-r}}{r}\dd r+o(1),
\end{equation}
the last integral being absolutely convergent.
\end{lemma}

\begin{proof}
We write $\sin^2(tr)=\tfrac12\bigl(1-\cos(2tr)\bigr)$, so that
\[
2\int_0^\infty\frac{\sin^2(tr)}{r}g(r)\dd r
=\int_0^\infty\frac{1-\cos(2tr)}{r}\,g(r)\dd r
=:\mathcal I_1(t)+\mathcal I_2(t),
\]
where, with $a:=g(0^+)$,
\[
\mathcal I_1(t):=a\int_0^\infty\frac{1-\cos(2tr)}{r}e^{-r}\dd r,
\qquad
\mathcal I_2(t):=\int_0^\infty\frac{1-\cos(2tr)}{r}
\bigl(g(r)-ae^{-r}\bigr)\dd r .
\]
For $\mathcal I_1$, we use the classical formula, valid for every $\lambda\in\R$,
\begin{equation}\label{eq:frullani}
\int_0^\infty\frac{1-\cos(\lambda r)}{r}e^{-r}\dd r
=\tfrac12\log\bigl(1+\lambda^2\bigr).
\end{equation}
It follows by taking real parts in the Frullani identity $\int_0^\infty r^{-1}\bigl(e^{-r}-e^{-(1-i\lambda)r}\bigr)\dd r=\log(1-i\lambda)$ {(alternatively, both sides of \eqref{eq:frullani} vanish at $\lambda=0$ and have the same derivative $\lambda/(1+\lambda^2)$)}. With $\lambda=2t$, this gives
\[
\mathcal I_1(t)=\frac a2\log\bigl(1+4t^2\bigr)
=a\log t+a\log2+o(1),\qquad t\to\infty .
\]
For $\mathcal I_2$, we set $\Psi(r):=r^{-1}\bigl(g(r)-ae^{-r}\bigr)$. Near the origin, $g(r)-ae^{-r}=\bigl(g(r)-a\bigr)+a\bigl(1-e^{-r}\bigr)$, so that \eqref{eq:dini} and $\abs{1-e^{-r}}\le r$ give $\Psi\in L^1(0,1)$. On $(1,\infty)$, we have $\abs{\Psi(r)}\le g(r)r^{-1}+\abs{a}e^{-r}r^{-1}$, which is integrable. Hence $\Psi\in L^1(0,\infty)$ and, by the Riemann--Lebesgue lemma,
\[
\mathcal I_2(t)=\int_0^\infty\Psi(r)\dd r
-\int_0^\infty\Psi(r)\cos(2tr)\dd r
=\int_0^\infty\Psi(r)\dd r+o(1).
\]
Adding the two contributions and dividing by $2$, we obtain \eqref{eq:log2}.
\end{proof}

\begin{remark}\rm[Normalization of the constant]\label{rem:normalization}
The reference profile $e^{-r}$ in Lemma~\ref{lem:log2} is not canonical. It may be replaced by any measurable function $\varphi$ satisfying
$$
\varphi(0^+)=1, \qquad \frac{\varphi(r)-1}{r}\in L^1(0,1), \qquad \frac{\varphi(r)}{r}\in L^1(1,\infty).
$$
The two constant terms in \eqref{eq:log2} then change separately, but their sum remains the same. We use $\varphi(r)=e^{-r}$ because the corresponding reference integral \eqref{eq:frullani} is explicit.
\end{remark}

Finally, we record the elementary inequality
\begin{equation}\label{eq:ab}
\abs{a+b}^2 \geq \frac12\abs a^2-\abs b^2, \qquad a,b\in\mathbb C,
\end{equation}
{which follows from $\abs a^2\le2\abs{a+b}^2+2\abs b^2$}. It will be used to isolate the contribution of the initial velocity in the proof of the unbounded part of Theorem~\ref{th:dichotomy}.

\section{The boundedness dichotomy}\label{sec:dichotomy}

\begin{proof}[Proof of Theorem \ref{th:dichotomy}]
We use the Fourier representation \eqref{eq:fourier-solution} with $s=0$ and abbreviate $\mathcal A=\mathcal A_0$, $\mathcal B=\mathcal B_0$ and $\mathcal D=\mathcal D_0$.

\smallskip
\noindent{Step 1: Sufficiency and the limit \eqref{eq:equipartition}.} Assume that $u_1\in\dH^{-1}(\R^N)$, and set
$$
F_1(\xi):= \frac{\abs{\hat u_1(\xi)}^2}{\abs{\xi}^2}, \qquad F_0(\xi):= \abs{\hat u_0(\xi)}^2.
$$
Then $F_1,F_0\in L^1(\R^N)$. Moreover, the inequality $2ab\le a^2+b^2$ gives
$$
\frac{ \abs{\hat u_1(\xi)} \abs{\hat u_0(\xi)} }{ \abs{\xi} } \leq \frac12\left(F_1(\xi)+F_0(\xi)\right),
$$
and hence $\hat u_1(\xi)\overline{\hat u_0(\xi)}/\abs{\xi}\in L^1(\R^N)$. Using $\sin^2(t\abs{\xi}) = \frac12\left(1-\cos(2t\abs{\xi})\right)$ and Lemma~\ref{lem:RL}, we obtain
$$
\mathcal A(t) = \frac12\int_{\R^N}F_1(\xi)\,\dd\xi - \frac12\int_{\R^N} F_1(\xi)\cos(2t\abs{\xi})\,\dd\xi\longrightarrow \frac12\int_{\R^N}F_1(\xi)\,\dd\xi.
$$
Similarly, since $\cos^2(t\abs{\xi}) = \frac12\left(1+\cos(2t\abs{\xi})\right)$, we get
$$
\mathcal B(t) \longrightarrow \frac12\int_{\R^N}F_0(\xi)\,\dd\xi.
$$
Another application of Lemma~\ref{lem:RL} gives $\mathcal D(t)\to0$. Dividing \eqref{eq:expansion} by $(2\pi)^N$ and using Plancherel's identity, we conclude that
$$
\lim_{t\to\infty}\norm{u(t,\cdot)}^2 = \frac12\left( \norm{u_1}_{\dH^{-1}}^2+\norm{u_0}^2 \right),
$$
which is \eqref{eq:equipartition}. The uniform bound can also be read off directly from \eqref{eq:fourier-solution}. Indeed, since $\abs{\sin(t\abs{\xi})}\leq1$ and $\abs{\cos(t\abs{\xi})}\leq1$, we have
$$
\norm{u(t,\cdot)} \leq \norm{ \abs{\nabla}^{-1}\sin(t\abs{\nabla})u_1 } + \norm{ \cos(t\abs{\nabla})u_0 } \leq \norm{u_1}_{\dH^{-1}}+\norm{u_0},
$$
and therefore $\sup_{t\geq0}\norm{u(t,\cdot)}<\infty$.

\smallskip
\noindent{Step 2: Necessity.} Assume that $u_1\notin\dH^{-1}(\R^N)$, so that
$$
\int_{\R^N}F_1(\xi)\,\dd\xi=+\infty, \qquad F_1(\xi):= \frac{\abs{\hat u_1(\xi)}^2}{\abs{\xi}^2}.
$$
We apply \eqref{eq:ab} pointwise to \eqref{eq:fourier-solution}, with
$$
a= \frac{\sin(t\abs{\xi})}{\abs{\xi}}\hat u_1(\xi), \qquad b= \cos(t\abs{\xi})\hat u_0(\xi),
$$
and we integrate over $\R^N$. We obtain
\begin{equation}\label{eq:lower-crude}
(2\pi)^N\norm{u(t,\cdot)}^2 \geq \frac12\mathcal A(t)-\norm{\hat u_0}_{L^2}^2.
\end{equation}
For $m>0$, we define $F_1^{(m)}(\xi) := \min\{F_1(\xi),m\} \mathbf 1_{\{\abs{\xi}\leq m\}}$. Then $F_1^{(m)}\in L^1(\R^N)$, $0\leq F_1^{(m)}\leq F_1$, and $F_1^{(m)}\uparrow F_1$ almost everywhere as $m\to\infty$. The monotone convergence theorem therefore gives
$$
\int_{\R^N}F_1^{(m)}(\xi)\,\dd\xi \longrightarrow+\infty \qquad\text{as }m\to\infty.
$$
Since $F_1\geq F_1^{(m)}\geq0$, we have
$$
\mathcal A(t) \geq \int_{\R^N} \sin^2(t\abs{\xi})F_1^{(m)}(\xi)\,\dd\xi.
$$
For each fixed $m>0$, the identity $\sin^2(t\abs{\xi}) = \frac12\left(1-\cos(2t\abs{\xi})\right)$ and Lemma~\ref{lem:RL} imply that
$$
\begin{aligned}
\int_{\R^N} \sin^2(t\abs{\xi})F_1^{(m)}(\xi)\,\dd\xi &= \frac12\int_{\R^N}F_1^{(m)}(\xi)\,\dd\xi- \frac12\int_{\R^N} F_1^{(m)}(\xi)\cos(2t\abs{\xi})\,\dd\xi\\
&\longrightarrow \frac12\int_{\R^N}F_1^{(m)}(\xi)\,\dd\xi
\end{aligned}
$$
as $t\to\infty$. Consequently,
$$
\liminf_{t\to\infty}\mathcal A(t) \geq \frac12\int_{\R^N}F_1^{(m)}(\xi)\,\dd\xi
$$
for every $m>0$. Letting $m\to\infty$, we deduce that $\mathcal A(t)\to+\infty$ as $t\to\infty$. It now follows from \eqref{eq:lower-crude} that
$$
\norm{u(t,\cdot)}\longrightarrow+\infty \qquad\text{as }t\to\infty.
$$
This proves both the necessity of the condition $u_1\in\dH^{-1}(\R^N)$ and the divergence in the complementary case.
\end{proof}

\begin{remark}\rm\label{rem:P-neq-0}
Suppose that $u_1\in L^1(\R^N)\cap L^2(\R^N)$. Then $\hat u_1$ is continuous and $\hat u_1(0)=P$. If $P\neq0$, continuity implies that $\abs{\hat u_1(\xi)}$ is bounded away from zero in a neighborhood of the origin. When $N\leq2$, it follows that
$$
\int_{\abs{\xi}\leq\rho} \frac{\abs{\hat u_1(\xi)}^2}{\abs{\xi}^2}\,\dd\xi = +\infty
$$
for every {$\rho>0$}. Thus $u_1\notin\dH^{-1}(\R^N)$ and, by Theorem~\ref{th:dichotomy}, $M(t)\to+\infty$ as $t\to\infty$. For $N\geq3$, on the contrary, $L^1(\R^N)\cap L^2(\R^N) \subset\dH^{-1}(\R^N)$, and hence $M(t)$ is uniformly bounded. This recovers the qualitative dimensional dichotomy underlying \cite[Theorems~1.1--1.2]{Ikehata2023}, without the weighted first-moment condition $u_1\in L^{1,1}(\R^N)$.
\end{remark}

\begin{remark}\rm[No universal divergence rate]\label{rem:norate}
Theorem~\ref{th:dichotomy} asserts that $M(t)\to+\infty$ whenever $u_1\notin\dH^{-1}(\R^N)$, but it gives no quantitative rate. No universal rate is possible under the energy-space assumptions alone, because the divergence is determined by the precise behavior of $\hat u_1(\xi)$ near $\xi=0$. In dimension one, a profile $\abs{\hat u_1(\xi)}\approx\abs{\xi}^{\alpha}$ near the origin with $0\leq\alpha<1/2$ produces $M(t)\asymp t^{1/2-\alpha}$, while the critical order $\alpha=1/2$ may produce logarithmic growth. Slowly varying corrections allow divergence slower than any positive power; see Problem~\ref{pb:RV}. Thus membership in $\dH^{-1}$ decides between boundedness and divergence, but not the rate. Under the additional hypotheses of Theorem~\ref{th:profiles}, the condition $P\neq0$ determines the exact first-order growth.
\end{remark}

\section{Exact first-order profiles}\label{sec:profiles}

We first establish the closed form \eqref{eq:kappa} of the constant $\kappa_{N,s}$.

\begin{lemma}\label{lem:kappa}
Let $N\ge1$ and $\tfrac{N-2}{2}<s<\tfrac N2$, and set $\mu:=3+2s-N\in(1,3)$. Then the integral
\[
\kappa_{N,s}=\int_{\R^N}\frac{\sin^2\abs{\eta}}{\abs{\eta}^{2+2s}}\dd\eta
\]
in \eqref{eq:kappa} converges, and
\[
\kappa_{N,s}=\sigma_{N-1}\int_0^\infty\frac{\sin^2r}{r^{\mu}}\dd r
=-\frac{2^{\mu-3}\pi\,\sigma_{N-1}}{\Gamma(\mu)\cos(\pi\mu/2)}
=\frac{2^{\mu-3}\pi\,\sigma_{N-1}}{\Gamma(\mu)\abs{\cos(\pi\mu/2)}} .
\]
\end{lemma}

\begin{proof}
By polar coordinates and the identity $N-3-2s=-\mu$, we have
$$
\kappa_{N,s} = \sigma_{N-1}\int_{0}^{\infty} \sin^{2}r\,r^{N-3-2s}\,\dd r = \sigma_{N-1}\int_{0}^{\infty} \frac{\sin^{2}r}{r^{\mu}}\,\dd r.
$$
As $r\to0$, the estimate $\sin^{2}r=O(r^{2})$ gives $r^{-\mu}\sin^{2}r=O(r^{2-\mu})$. Since $\mu<3$, we have $2-\mu>-1$, and the integral converges near $r=0$. On the other hand, $0\leq r^{-\mu}\sin^{2}r\leq r^{-\mu}$, and since $\mu>1$, the integral also converges at infinity. Using $\sin^{2}r=\frac12(1-\cos(2r))$ and the change of variables $y=2r$, we obtain
$$
\int_{0}^{\infty}\frac{\sin^{2}r}{r^{\mu}}\,\dd r = \frac12\int_{0}^{\infty} \frac{1-\cos(2r)}{r^{\mu}}\,\dd r = 2^{\mu-2}\int_{0}^{\infty} \frac{1-\cos y}{y^{\mu}}\,\dd y.
$$
For $-2<\nu<0$, the classical identity
$$
\int_{0}^{\infty}y^{\nu-1}(\cos y-1)\,\dd y = \Gamma(\nu)\cos\left(\frac{\pi\nu}{2}\right)
$$
holds, the right-hand side being understood by continuous extension at $\nu=-1$ {(it follows from the formula $\int_0^\infty y^{\nu-1}\cos y\,\dd y=\Gamma(\nu)\cos(\pi\nu/2)$, $0<\nu<1$, by analytic continuation; see, e.g., \cite[Section~3.76]{GR2007})}. Taking $\nu=1-\mu\in(-2,0)$, we get
$$
\int_{0}^{\infty}\frac{1-\cos y}{y^{\mu}}\,\dd y = -\Gamma(1-\mu) \cos\left(\frac{\pi(1-\mu)}{2}\right).
$$
Since $\cos\left(\frac{\pi(1-\mu)}{2}\right) =\sin\left(\frac{\pi\mu}{2}\right)$, the reflection formula $\Gamma(\mu)\Gamma(1-\mu) = \pi/\sin(\pi\mu)$ yields, for $\mu\neq2$,
$$
\int_{0}^{\infty}\frac{1-\cos y}{y^{\mu}}\,\dd y = -\frac{\pi\sin(\pi\mu/2)} {\Gamma(\mu)\sin(\pi\mu)} = -\frac{\pi} {2\Gamma(\mu)\cos(\pi\mu/2)}.
$$
Both sides extend continuously to $\mu=2$, where $\int_{0}^{\infty}y^{-2}(1-\cos y)\,\dd y = \pi/2$. Thus the identity holds for every $\mu\in(1,3)$, and we deduce that
$$
\int_{0}^{\infty}\frac{\sin^{2}r}{r^{\mu}}\,\dd r = -\frac{2^{\mu-3}\pi} {\Gamma(\mu)\cos(\pi\mu/2)}.
$$
Multiplying by $\sigma_{N-1}$, we obtain the first closed form of $\kappa_{N,s}$. Finally, since $\mu\in(1,3)$, we have $\frac{\pi\mu}{2}\in(\frac{\pi}{2},\frac{3\pi}{2})$, and hence $\cos(\pi\mu/2)<0$. This gives the second closed form and completes the proof.
\end{proof}

\begin{remark}\rm
For $N=1$ and $s=0$, we have $\mu=2$, $\sigma_0=2$ and $\kappa_{1,0}=\pi$, consistently with $\int_\R\eta^{-2}\sin^2\eta\,\dd\eta=\pi$. For $N=3$ and $s=1$, we find $\mu=2$ and $\kappa_{3,1}=2\pi^2$. Note that $\kappa_{N,s}\to+\infty$ as $s\downarrow s_c$ or $s\uparrow N/2$, in agreement with the degeneration of the polynomial regime at both endpoints.
\end{remark}

\begin{proof}[Proof of Theorem \ref{th:profiles}]
Since $u_1\in L^1(\R^N)$, its Fourier transform $\widehat u_1$ is bounded and continuous. Moreover, $\left|\widehat u_1(\xi)\right| \leq \norm{u_1}_{L^1}$ for every $\xi\in\R^N$, and $\widehat u_1(0) = P$.

\smallskip
\noindent{Step 1: The term $\mathcal B_s$ is bounded.} Since $0\leq s<\frac{N}{2}$ and $u_0\in\dH^{-s}(\R^N)$, we have
\begin{equation}\label{eq:B-bounded}
0 \leq \mathcal B_s(t) \leq \int_{\R^N} |\xi|^{-2s} |\widehat u_0(\xi)|^2 \,\dd\xi = (2\pi)^N \norm{u_0}_{\dH^{-s}}^2 < \infty
\end{equation}
for every $t\geq0$. In particular, $\mathcal B_s(t)=O(1)$ as $t\to\infty$.

If one assumes in addition that $u_0\in L^1(\R^N)$, then the condition $u_0\in\dH^{-s}(\R^N)$ is automatic. Indeed, we have
$$
\int_{|\xi|\leq1} |\xi|^{-2s} |\widehat u_0(\xi)|^2 \,\dd\xi \leq \norm{u_0}_{L^1}^2 \int_{|\xi|\leq1}|\xi|^{-2s}\,\dd\xi < \infty,
$$
because $2s<N$, whereas
$$
\int_{|\xi|\geq1} |\xi|^{-2s} |\widehat u_0(\xi)|^2 \,\dd\xi \leq \norm{\widehat u_0}_{L^2}^2 < \infty.
$$

\smallskip
\noindent{Step 2: The polynomial regime $\theta>0$.} The condition $\theta>0$ is equivalent to $\frac{N-2}{2}<s<\frac{N}{2}$. The change of variables $\xi=\eta/t$ in \eqref{eq:A} gives, for every $t>0$,
\begin{equation}\label{eq:scaling}
\mathcal A_s(t) = t^\theta \int_{\R^N} \frac{\sin^2|\eta|}{|\eta|^{2+2s}} \left| \widehat u_1\left(\frac{\eta}{t}\right) \right|^2 \,\dd\eta .
\end{equation}
Indeed, the scaling factor is $t^{2+2s-N}=t^\theta$. For every $\eta\in\R^N$ and $t>0$, we have $\left| \widehat u_1(\eta/t) \right|^2 \leq \norm{u_1}_{L^1}^2$. {Therefore, the integrand in \eqref{eq:scaling} is dominated by $\norm{u_1}_{L^1}^2\abs{\eta}^{-2-2s}\sin^2\abs{\eta}$, uniformly in $t>0$.} By Lemma~\ref{lem:kappa}, this dominating function belongs to $L^1(\R^N)$ in the range $\frac{N-2}{2}<s<\frac{N}{2}$. For every fixed $\eta\in\R^N$, the continuity of $\widehat u_1$ at the origin gives $\widehat u_1(\eta/t) \to \widehat u_1(0) = P$ as $t\to\infty$, and consequently $\left| \widehat u_1(\eta/t) \right|^2 \to |P|^2$. The dominated convergence theorem therefore yields
$$
t^{-\theta}\mathcal A_s(t) \longrightarrow |P|^2 \int_{\R^N} \frac{\sin^2|\eta|}{|\eta|^{2+2s}} \,\dd\eta = \kappa_{N,s}|P|^2,
$$
that is,
\begin{equation}\label{eq:A-poly}
\mathcal A_s(t) = \kappa_{N,s}|P|^2t^\theta + o(t^\theta).
\end{equation}
By Step~1, $\mathcal B_s(t)=O(1)=o(t^\theta)$. Moreover, the Cauchy--Schwarz estimate \eqref{eq:cross-CS} gives
$$
|\mathcal D_s(t)| \leq {2}\,\mathcal A_s(t)^{1/2} \mathcal B_s(t)^{1/2}.
$$
Since $\mathcal A_s(t)=O(t^\theta)$ and $\mathcal B_s(t)=O(1)$, we obtain $\mathcal D_s(t) = O(t^{\theta/2}) = o(t^\theta)$. Substituting these estimates into \eqref{eq:expansion}, we conclude that
$$
\norm{u(t,\cdot)}_{\dH^{-s}}^2 = \frac{\kappa_{N,s}}{(2\pi)^N} |P|^2t^\theta + o(t^\theta),
$$
which is equivalent to part~{\rm (i)}.

\smallskip
\noindent{Step 3: The logarithmic regime $\theta=0$.} In this case, $s=s_c=\frac{N-2}{2}$. Since $s\geq0$, this regime occurs only when $N\geq2$. Recall that
$$
g(r) = \int_{\mathbb S^{N-1}} |\widehat u_1(r\omega)|^2 \,\dd\omega.
$$
Passing to polar coordinates in \eqref{eq:A}, we obtain
$$
\mathcal A_{s_c}(t) = \int_0^\infty \sin^2(tr)\, r^{N-1-(2s_c+2)} g(r)\,\dd r.
$$
Since $N-1-(2s_c+2)=-1$, it follows that
\begin{equation}\label{eq:A-polar}
\mathcal A_{s_c}(t) = \int_0^\infty \frac{\sin^2(tr)}{r} g(r)\,\dd r.
\end{equation}
The function $g$ is measurable and bounded, since $0\leq g(r) \leq \sigma_{N-1}\norm{u_1}_{L^1}^2$. Furthermore, $\widehat u_1(r\omega) \to P$ as $r\downarrow0$ for every $\omega\in\mathbb S^{N-1}$. By dominated convergence on $\mathbb S^{N-1}$, we get $g(r)\to\sigma_{N-1}|P|^2$, that is, $g(0^+)=\sigma_{N-1}|P|^2$. Since $N\geq2$, we have $r^{-1}\leq r^{N-1}$ for $r\geq1$. Consequently,
\begin{equation}\label{eq:g-tail}
\int_1^\infty\frac{g(r)}{r}\,\dd r \leq \int_1^\infty g(r)r^{N-1}\,\dd r \leq \int_0^\infty g(r)r^{N-1}\,\dd r = \int_{\R^N} |\widehat u_1(\xi)|^2 \,\dd\xi = \norm{\widehat u_1}_{L^2}^2 < \infty.
\end{equation}
All the assumptions of Lemma~\ref{lem:log} are therefore satisfied. Applying that lemma to \eqref{eq:A-polar}, we obtain
$$
\mathcal A_{s_c}(t) = \frac12g(0^+)\log t + o(\log t) = \frac{\sigma_{N-1}}{2} |P|^2\log t + o(\log t).
$$
By Step~1, $\mathcal B_{s_c}(t) = O(1) = o(\log t)$. Moreover, \eqref{eq:cross-CS} implies that
$$
|\mathcal D_{s_c}(t)| \leq {2}\,\mathcal A_{s_c}(t)^{1/2} \mathcal B_{s_c}(t)^{1/2} = O\bigl(\sqrt{\log t}\bigr) = o(\log t).
$$
Using \eqref{eq:expansion}, we conclude that
$$
\norm{u(t,\cdot)}_{\dH^{-s_c}}^2 = \frac{\sigma_{N-1}}{2(2\pi)^N} |P|^2\log t + o(\log t),
$$
which proves part~{\rm (ii)}.

\smallskip
\noindent{Step 4: The convergent regime $\theta<0$.} The condition $\theta<0$ is equivalent to $2s+2<N$. Using the boundedness of $\widehat u_1$, we obtain
$$
\int_{|\xi|\leq1} \frac{|\widehat u_1(\xi)|^2}{|\xi|^{2s+2}} \,\dd\xi \leq \norm{u_1}_{L^1}^2 \int_{|\xi|\leq1} |\xi|^{-2s-2} \,\dd\xi < \infty,
$$
because $2s+2<N$. On the other hand, since $|\xi|^{-2s-2}\leq1$ for $|\xi|\geq1$, we have
$$
\int_{|\xi|\geq1} \frac{|\widehat u_1(\xi)|^2}{|\xi|^{2s+2}} \,\dd\xi \leq \int_{|\xi|\geq1} |\widehat u_1(\xi)|^2 \,\dd\xi \leq \norm{\widehat u_1}_{L^2}^2 < \infty.
$$
It follows that $u_1\in\dH^{-s-1}(\R^N)$, while $u_0\in\dH^{-s}(\R^N)$ by assumption. Using $\sin^2(t|\xi|) = \frac12 \left( 1-\cos(2t|\xi|) \right)$, we obtain
$$
\mathcal A_s(t) = \frac12 \int_{\R^N} |\xi|^{-2s-2} |\widehat u_1(\xi)|^2 \,\dd\xi - \frac12 \int_{\R^N} |\xi|^{-2s-2} |\widehat u_1(\xi)|^2 \cos(2t|\xi|) \,\dd\xi.
$$
Similarly, using $\cos^2(t|\xi|) = \frac12 \left( 1+\cos(2t|\xi|) \right)$, we get
$$
\mathcal B_s(t) = \frac12 \int_{\R^N} |\xi|^{-2s} |\widehat u_0(\xi)|^2 \,\dd\xi + \frac12 \int_{\R^N} |\xi|^{-2s} |\widehat u_0(\xi)|^2 \cos(2t|\xi|) \,\dd\xi.
$$
The functions $F_1(\xi) := |\xi|^{-2s-2} |\widehat u_1(\xi)|^2$ and $F_0(\xi) := |\xi|^{-2s} |\widehat u_0(\xi)|^2$ belong to $L^1(\R^N)$. Therefore, Lemma~\ref{lem:RL} yields
$$
\int_{\R^N} F_1(\xi)\cos(2t|\xi|) \,\dd\xi \longrightarrow0
\qquad\text{and}\qquad
\int_{\R^N} F_0(\xi)\cos(2t|\xi|) \,\dd\xi \longrightarrow0.
$$
Consequently,
$$
\mathcal A_s(t) \longrightarrow \frac{(2\pi)^N}{2} \norm{u_1}_{\dH^{-s-1}}^2
\qquad\text{and}\qquad
\mathcal B_s(t) \longrightarrow \frac{(2\pi)^N}{2} \norm{u_0}_{\dH^{-s}}^2.
$$
For the mixed term, we define $F_{01}(\xi) := |\xi|^{-2s-1} \operatorname{Re} \left( \widehat u_1(\xi) \overline{\widehat u_0(\xi)} \right)$. By Young's inequality, we have
$$
|F_{01}(\xi)| \leq |\xi|^{-2s-1} |\widehat u_1(\xi)| |\widehat u_0(\xi)| \leq \frac12 |\xi|^{-2s-2} |\widehat u_1(\xi)|^2 + \frac12 |\xi|^{-2s} |\widehat u_0(\xi)|^2 .
$$
The right-hand side belongs to $L^1(\R^N)$, and hence $F_{01}\in L^1(\R^N)$. Since
$$
\mathcal D_s(t)=\int_{\R^N}F_{01}(\xi)\sin(2t\abs{\xi})\,\dd\xi,
$$
Lemma~\ref{lem:RL} gives $\mathcal D_s(t)\to0$. Finally, substituting the limits of $\mathcal A_s$, $\mathcal B_s$ and $\mathcal D_s$ into \eqref{eq:expansion}, we obtain
$$
\lim_{t\to\infty} \norm{u(t,\cdot)}_{\dH^{-s}}^2 = \frac12 \left( \norm{u_1}_{\dH^{-s-1}}^2 + \norm{u_0}_{\dH^{-s}}^2 \right).
$$
This proves part~{\rm (iii)} and completes the proof.
\end{proof}

\begin{proof}[Proof of Corollary \ref{cor:L2}]
Take $s=0$, so that $\theta=2-N$ and $\dH^{-s}=L^2$. {In particular, the assumption $u_0\in\dH^{0}(\R^N)=L^2(\R^N)$ of Theorem~\ref{th:profiles} is satisfied.} For $N=1$, we have $\theta=1>0$ and $\kappa_{1,0}/(2\pi)=\pi/(2\pi)=1/2$. For $N=2$, we have $\theta=0$ and $\sigma_1/(2(2\pi)^2)=2\pi/(8\pi^2)=1/(4\pi)$. For $N\ge3$, we have $\theta<0$, and part~(iii) of Theorem~\ref{th:profiles} reduces to \eqref{eq:equipartition}.
\end{proof}

\begin{remark}\rm\label{rem:sharp-const}
{Corollary~\ref{cor:L2} shows that, when $P\neq0$, the two-sided bounds of \cite[Theorems~1.1--1.2]{Ikehata2023} collapse into the asymptotic equalities
$$
M(t)=\frac{\abs{P}}{\sqrt2}\,\sqrt t\,\bigl(1+o(1)\bigr)\quad(N=1),
\qquad
M(t)=\frac{\abs{P}}{2\sqrt\pi}\,\sqrt{\log t}\,\bigl(1+o(1)\bigr)\quad(N=2).
$$
In particular, the optimal constant $C_2$ in the lower bounds of \eqref{eq:ike1} and \eqref{eq:ike2} is $1/\sqrt2$ and $1/(2\sqrt\pi)$, respectively, and the upper bounds hold asymptotically with $C_1I_{0,N}$ replaced by the same multiple of $\abs{P}$.} It also shows that, at leading order, the growth is insensitive to $u_0$ and to all moments of $u_1$ of positive order. Only the total mass $P$ survives.
\end{remark}

\begin{remark}\rm[Higher dimensions]\label{rem:higher-d}
Let $N\geq3$ and $s=s_c=\frac{N-2}{2}$. Then Theorem~\ref{th:profiles}(ii) gives
$$
\norm{u(t,\cdot)}_{\dH^{-(N-2)/2}}^2 = \frac{\sigma_{N-1}}{2(2\pi)^N} |P|^2\log t + o(\log t)
$$
as $t\to\infty$. In particular, when $N=3$, we have $s_c=\frac12$ and $\sigma_2=4\pi$, and hence
$$
\norm{u(t,\cdot)}_{\dH^{-1/2}}^2 = \frac{|P|^2}{4\pi^2}\log t + o(\log t).
$$
When $N=4$, we have $s_c=1$ and $\sigma_3=2\pi^2$, so that
$$
\norm{u(t,\cdot)}_{\dH^{-1}}^2 = \frac{|P|^2}{16\pi^2}\log t + o(\log t).
$$
Thus, whenever $P\neq0$, the logarithmic growth found in dimension two in \cite{Ikehata2023} also occurs in every dimension $N\geq3$, provided that the solution is measured in the scaling-critical homogeneous Sobolev space $\dH^{-(N-2)/2}(\R^N)$. Example~\ref{ex:3dgauss} makes this explicit for Gaussian data in $\R^3$. There, the $L^2$-norm converges, whereas the $\dH^{-1/2}$-norm diverges logarithmically.
\end{remark}

\begin{remark}\rm[Interpolation is lossy]\label{rem:interp}
It is tempting to deduce the logarithmic regime by interpolating between the polynomial and the convergent ones. This cannot work. Interpolating the bounds $\norm{u(t)}_{\dH^{-s}}^2\lesssim t^{\theta(s)}$ for $s>s_c$ and $\norm{u(t)}_{\dH^{-s}}^2\lesssim 1$ for $s<s_c$ gives, at $s=s_c$, only $\norm{u(t)}_{\dH^{-s_c}}^2\lesssim_\varepsilon t^{\varepsilon}$ for every $\varepsilon>0$. The logarithm, and a fortiori its constant, is a genuine boundary effect, which requires the direct argument of Lemma~\ref{lem:log}.
\end{remark}

\section{Second-order expansions}\label{sec:secondorder}

We begin this section with an integrability property of $\abs{\hat u_1}^2-P^2$ near the origin in dimension one.

\begin{lemma}\label{lem:taylor}
Let $N=1$, let $u_1\in L^{1,1}(\R)$ be real-valued, and set
$$
P:=\int_\R u_1(x)\,\dd x,\qquad m_1:=\int_\R\abs{x}\abs{u_1(x)}\,\dd x .
$$
Then the function $F(\xi):=\xi^{-2}\bigl(|\widehat u_1(\xi)|^2-P^2\bigr)$ belongs to $L^1(\R)$, and
\begin{equation}\label{eq:F-L1}
\norm{F}_{L^1(\R)}\le 2\pi\norm{u_1}_{L^1}m_1+2m_1^2+4\norm{u_1}_{L^1}^2 .
\end{equation}
If, in addition, $u_1\in L^{1,2}(\R)$, then $\abs{F(\xi)}\le \norm{u_1}_{L^1}m_2+m_1^2$ for every $\xi\neq0$, where $m_2:=\int_\R x^2\abs{u_1(x)}\,\dd x$.
\end{lemma}

\begin{proof}
Since $u_1$ is real-valued, we may write $\widehat u_1(\xi) = X(\xi)-iY(\xi)$, where
$$
X(\xi) := \int_\R \cos(x\xi)u_1(x)\,\dd x
\qquad\text{and}\qquad
Y(\xi) := \int_\R \sin(x\xi)u_1(x)\,\dd x
$$
are real. Hence $|\widehat u_1(\xi)|^2-P^2 = \bigl(X(\xi)-P\bigr)\bigl(X(\xi)+P\bigr)+Y(\xi)^2$, with $\abs{X(\xi)+P}\le2\norm{u_1}_{L^1}$.

We first consider $\abs{\xi}\le1$. Since
$$
\abs{X(\xi)-P}=\left| \int_\R \bigl(\cos(x\xi)-1\bigr) u_1(x)\,\dd x \right| \le \int_\R \bigl(1-\cos(x\xi)\bigr)\abs{u_1(x)}\,\dd x,
$$
Tonelli's theorem gives
$$
\int_{\abs{\xi}\le1}\frac{\abs{X(\xi)-P}}{\xi^2}\,\dd\xi
\le\int_\R\abs{u_1(x)}\left(\int_{-1}^{1}\frac{1-\cos(x\xi)}{\xi^2}\,\dd\xi\right)\dd x .
$$
For $x\neq0$, the change of variables $y=\abs{x}\xi$ shows that the inner integral equals
$$
2\abs{x}\int_0^{\abs{x}}\frac{1-\cos y}{y^2}\,\dd y\le2\abs{x}\int_0^{\infty}\frac{1-\cos y}{y^2}\,\dd y=\pi\abs{x}.
$$
We infer that $\int_{\abs{\xi}\le1}\xi^{-2}\abs{X(\xi)^2-P^2}\,\dd\xi\le2\pi\norm{u_1}_{L^1}m_1$. Moreover, $\abs{\sin\vartheta}\le\abs\vartheta$ gives $\abs{Y(\xi)}\le\abs{\xi}m_1$, so that $\int_{\abs{\xi}\le1}\xi^{-2}Y(\xi)^2\,\dd\xi\le2m_1^2$. For $\abs{\xi}\ge1$, we use $\abs{\widehat u_1(\xi)}\le\norm{u_1}_{L^1}$ and $\abs P\le\norm{u_1}_{L^1}$, and we get $\abs{F(\xi)}\le2\norm{u_1}_{L^1}^2\xi^{-2}$, whose integral over $\{\abs\xi\ge1\}$ equals $4\norm{u_1}_{L^1}^2$. Adding these three bounds, we obtain \eqref{eq:F-L1}.

If $u_1\in L^{1,2}(\R)$, the inequality $1-\cos\vartheta\le\vartheta^2/2$ gives $\abs{X(\xi)-P}\le\frac{\xi^2}{2}m_2$, and hence $\abs{X(\xi)^2-P^2}\le\xi^2\norm{u_1}_{L^1}m_2$. Together with $Y(\xi)^2\le\xi^2m_1^2$, this yields the pointwise bound.
\end{proof}

\begin{remark}\rm[Sharpness of the moment condition]\label{rem:moment-sharp}
Since $u_1$ is real-valued, the linear term in the expansion of $|\widehat u_1(\xi)|^2$ at $\xi=0$ cancels. More precisely, if $u_1\in L^{1,2}(\R)$ and $M_1:=\int_{\R}x u_1(x)\,\dd x$, then $\widehat u_1(\xi) = P-iM_1\xi+O(\xi^2)$, and therefore $|\widehat u_1(\xi)|^2 = P^2+O(\xi^2)$. Lemma~\ref{lem:taylor} shows that the integrability of $\xi^{-2}\bigl(|\widehat u_1(\xi)|^2-P^2\bigr)$, which is all that the proof of Theorem~\ref{th:secondorder1} requires, holds under the weaker assumption $u_1\in L^{1,1}(\R)$. This assumption cannot be weakened in general. Indeed, let $u_1\ge0$ be even, with $P>0$. Then $Y\equiv0$ and $0\le P-X(\xi)\to0$ as $\xi\to0$, so that $\abs{F(\xi)}\ge P\,(P-X(\xi))\,\xi^{-2}$ for $\abs\xi\le\delta$ and some $\delta\in(0,1]$. The computation in the proof of Lemma~\ref{lem:taylor} then gives
$$
\int_{\abs\xi\le\delta}\abs{F(\xi)}\,\dd\xi\ \ge\ 2P\int_\R \abs{x}\,u_1(x)\left(\int_0^{\delta\abs x}\frac{1-\cos y}{y^2}\,\dd y\right)\dd x
\ \ge\ 2P\,c_\delta\int_{\abs x\ge1}\abs{x}\,u_1(x)\,\dd x,
$$
with $c_\delta:=\int_0^{\delta}y^{-2}(1-\cos y)\,\dd y>0$. Hence, for such data, $F\in L^1(\R)$ if and only if $u_1\in L^{1,1}(\R)$. Example~\ref{ex:cauchy}, for which $u_1\in L^{1,\alpha}(\R)$ for every $\alpha<1$ but $u_1\notin L^{1,1}(\R)$, shows that the conclusion of Theorem~\ref{th:secondorder1} itself fails in this borderline case: the second-order term is then $-\pi\log t$ instead of a constant.
\end{remark}

\begin{proof}[Proof of Theorem \ref{th:secondorder1}]
We use \eqref{eq:expansion} with $N=1$ and $s=0$.

\smallskip
\noindent{The term $\mathcal A_0$.} The change of variables $\eta=t\xi$ and the classical identity $\int_\R\eta^{-2}\sin^2\eta\,\dd\eta=\pi$ give
$$
\int_\R\frac{\sin^2(t\xi)}{\xi^2}\,\dd\xi = t\int_\R\frac{\sin^2\eta}{\eta^2}\,\dd\eta = \pi t.
$$
Since $\sin^2(t|\xi|)=\sin^2(t\xi)$, we get
$$
\mathcal A_0(t) = \int_\R \frac{\sin^2(t\xi)}{\xi^2} |\widehat u_1(\xi)|^2 \,\dd\xi = \pi P^2t + \int_\R \frac{\sin^2(t\xi)}{\xi^2} \left( |\widehat u_1(\xi)|^2-P^2 \right) \,\dd\xi.
$$
By Lemma~\ref{lem:taylor}, the function $F(\xi) := \xi^{-2}\bigl(|\widehat u_1(\xi)|^2-P^2\bigr)$ belongs to $L^1(\R)$. Since $\sin^2(t\xi) = \frac12 \bigl( 1-\cos(2t\xi) \bigr)$, we obtain
$$
\int_\R F(\xi)\sin^2(t\xi) \,\dd\xi = \frac12\int_\R F(\xi)\,\dd\xi - \frac12\int_\R F(\xi)\cos(2t\xi)\,\dd\xi.
$$
The Riemann--Lebesgue lemma gives $\int_\R F(\xi)\cos(2t\xi)\,\dd\xi \to0$. Consequently,
$$
\mathcal A_0(t) = \pi P^2t + \frac12 \int_\R \frac{|\widehat u_1(\xi)|^2-P^2}{\xi^2} \,\dd\xi + o(1).
$$

\smallskip
\noindent{The term $\mathcal B_0$.} Since $u_0\in L^2(\R)$, Plancherel's theorem gives $|\widehat u_0|^2\in L^1(\R)$. Using $\cos^2(t|\xi|) = \frac12 \bigl( 1+\cos(2t|\xi|) \bigr)$ and the fact that $\cos(2t\abs\xi)=\cos(2t\xi)$, we have
$$
\mathcal B_0(t) = \int_\R \cos^2(t|\xi|) |\widehat u_0(\xi)|^2 \,\dd\xi = \frac12 \norm{\widehat u_0}_{L^2}^2 + \frac12 \int_\R |\widehat u_0(\xi)|^2 \cos(2t\xi) \,\dd\xi.
$$
The Riemann--Lebesgue lemma yields $\int_\R |\widehat u_0(\xi)|^2 \cos(2t\xi) \,\dd\xi \to0$, and therefore $\mathcal B_0(t) = \frac12\norm{\widehat u_0}_{L^2}^2 + o(1)$. With our normalization, Plancherel's identity reads $\norm{\widehat u_0}_{L^2}^2 = 2\pi\norm{u_0}_{L^2}^2$, and we conclude that
$$
\mathcal B_0(t) = \pi\norm{u_0}_{L^2}^2 + o(1).
$$

\smallskip
\noindent{The term $\mathcal D_0$.} We define $h(\xi) := \operatorname{Re} \bigl( \widehat u_1(\xi) \overline{\widehat u_0(\xi)} \bigr)$. Since $u_0$ and $u_1$ are real-valued, $\widehat u_j(-\xi) = \overline{\widehat u_j(\xi)}$, and hence $h$ is real-valued and even. Moreover, $h(0) = \widehat u_1(0)\widehat u_0(0) = PQ$. Using $2\sin(t|\xi|)\cos(t|\xi|) = \sin(2t|\xi|)$, we obtain
$$
\mathcal D_0(t) = \int_\R \frac{\sin(2t|\xi|)}{|\xi|} h(\xi) \,\dd\xi = 2\int_0^\infty \frac{\sin(2t\xi)}{\xi} h(\xi) \,\dd\xi.
$$
Since {$u_1,u_0\in L^{1,1}(\R)$}, the Fourier transforms $\widehat u_1$ and $\widehat u_0$ are continuously differentiable with bounded derivatives. It follows that $h\in C^1(\R)$, so that $(h(\xi)-h(0))/\xi$ is bounded on $(0,1)$, hence integrable there. Furthermore, the Cauchy--Schwarz inequality gives
$$
\begin{aligned}
\int_1^\infty\frac{|h(\xi)|}{\xi}\,\dd\xi &\leq \int_1^\infty \frac{ |\widehat u_1(\xi)| |\widehat u_0(\xi)| }{\xi} \,\dd\xi \leq \norm{\widehat u_1}_{L^2(1,\infty)} \norm{\xi^{-1}\widehat u_0}_{L^2(1,\infty)}\\
&\leq \norm{\widehat u_1}_{L^2(\R)} \norm{\widehat u_0}_{L^2(\R)} < \infty.
\end{aligned}
$$
Thus the assumptions of Lemma~\ref{lem:dirichlet} are satisfied, and that lemma gives
$$
\int_0^\infty \frac{\sin(2t\xi)}{\xi} h(\xi) \,\dd\xi \longrightarrow \frac{\pi}{2}h(0).
$$
We deduce that $\mathcal D_0(t) \to \pi PQ$.

\smallskip
\noindent{Conclusion.} By \eqref{eq:expansion}, $\norm{u(t,\cdot)}_{L^2}^2 = \frac{1}{2\pi} \left( \mathcal A_0(t) + \mathcal B_0(t) + \mathcal D_0(t) \right)$. Substituting the preceding asymptotic formulas, we obtain
$$
\begin{aligned}
\norm{u(t,\cdot)}_{L^2}^2 &= \frac{1}{2\pi} \left[ \pi P^2t + \frac12 \int_\R \frac{|\widehat u_1(\xi)|^2-P^2}{\xi^2} \,\dd\xi + \pi\norm{u_0}_{L^2}^2 + \pi PQ \right] + o(1) \\
&= \frac{P^2}{2}\,t + \frac{1}{4\pi} \int_\R \frac{|\widehat u_1(\xi)|^2-P^2}{\xi^2} \,\dd\xi + \frac12\norm{u_0}_{L^2}^2 + \frac{PQ}{2} + o(1).
\end{aligned}
$$
This proves \eqref{eq:secondorder1}.
\end{proof}

We now turn to the proof of Theorem~\ref{th:secondorder2}.

\begin{proof}[Proof of Theorem \ref{th:secondorder2}]
We use \eqref{eq:expansion} with $s=s_c=\frac{N-2}{2}$. In particular, $2s_c+2=N$ and $2s_c=N-2$.

\smallskip
\noindent{The term $\mathcal A_{s_c}$.} By \eqref{eq:A-polar}, we have $\mathcal A_{s_c}(t) = \int_0^\infty r^{-1}\sin^2(tr)\, g(r)\,\dd r$, and we check the hypotheses of Lemma~\ref{lem:log2}. Since $u_1\in L^1(\R^N)$, its Fourier transform is bounded and continuous, with $\left|\widehat u_1(\xi)\right| \leq \norm{u_1}_{L^1}$ and $\widehat u_1(0)=P$. It follows that $g$ is measurable and bounded, with $0\leq g(r) \leq \sigma_{N-1}\norm{u_1}_{L^1}^2$. Moreover, dominated convergence on the sphere gives
$$
g(0^+) = \int_{\Sph^{N-1}} \left|\widehat u_1(0)\right|^2 \,\dd\omega = \sigma_{N-1}P^2.
$$
Since $u_1\in L^{1,1}(\R^N)$, we have
$$
\left|\widehat u_1(\xi)-P\right| = \left| \int_{\R^N} \left( e^{-ix\cdot\xi}-1 \right) u_1(x)\,\dd x \right| \leq |\xi| \int_{\R^N} |x||u_1(x)|\,\dd x.
$$
Using $\bigl| |\widehat u_1(\xi)|^2-P^2 \bigr| \leq \left|\widehat u_1(\xi)-P\right| \left( |\widehat u_1(\xi)|+|P| \right)$ and $|\widehat u_1(\xi)|+|P| \leq 2\norm{u_1}_{L^1}$, we obtain
$$
\bigl| |\widehat u_1(\xi)|^2-P^2 \bigr| \leq C|\xi|,
$$
where $C$ depends only on $\norm{u_1}_{L^1}$ and on the first absolute moment of $u_1$. Consequently, for $0<r<1$, we have
$$
|g(r)-g(0^+)| \leq \int_{\Sph^{N-1}} \bigl| |\widehat u_1(r\omega)|^2-P^2 \bigr| \,\dd\omega \leq C\sigma_{N-1}r .
$$
Hence $(g(r)-g(0^+))/r$ is bounded, and therefore integrable, on $(0,1)$. At infinity, the bound \eqref{eq:g-tail} gives $\int_1^\infty g(r)r^{-1}\,\dd r<\infty$. In particular, the integral $\int_0^\infty r^{-1}\bigl(g(r)-g(0^+)e^{-r}\bigr)\,\dd r$ converges absolutely, as shown in the proof of Lemma~\ref{lem:log2}. Lemma~\ref{lem:log2} now gives
$$
\mathcal A_{s_c}(t) = \frac{g(0^+)}{2}\log t + \frac{g(0^+)}{2}\log2 + \frac12 \int_0^\infty \frac{ g(r)-g(0^+)e^{-r} }{r} \,\dd r + o(1).
$$
Since $g(0^+)=\sigma_{N-1}P^2$, we conclude that
$$
\mathcal A_{s_c}(t) = \frac{\sigma_{N-1}P^2}{2}\log t + \frac{\sigma_{N-1}P^2}{2}\log2 + \frac12 \int_0^\infty \frac{ g(r)-\sigma_{N-1}P^2e^{-r} }{r} \,\dd r + o(1).
$$

\smallskip
\noindent{The term $\mathcal B_{s_c}$.} Since $2s_c=N-2<N$ and $u_0\in L^1(\R^N)\cap L^2(\R^N)$, Step~1 of the proof of Theorem~\ref{th:profiles} shows that $|\xi|^{-2s_c} |\widehat u_0(\xi)|^2 \in L^1(\R^N)$. Using $\cos^2(t|\xi|) = \frac12 \left( 1+\cos(2t|\xi|) \right)$ and Lemma~\ref{lem:RL}, we obtain
$$
\mathcal B_{s_c}(t) \longrightarrow \frac12 \int_{\R^N} |\xi|^{-2s_c} |\widehat u_0(\xi)|^2 \,\dd\xi = \frac{(2\pi)^N}{2} \norm{u_0}_{\dH^{-s_c}}^2 .
$$

\smallskip
\noindent{The term $\mathcal D_{s_c}$.} We set $h(\xi) := \operatorname{Re} \bigl( \widehat u_1(\xi) \overline{\widehat u_0(\xi)} \bigr)$ and
$$
H(r) := \int_{\Sph^{N-1}} h(r\omega)\,\dd\omega .
$$
Since $2s_c+1=N-1$, polar coordinates give
$$
\mathcal D_{s_c}(t) = \int_0^\infty r^{N-1-(N-1)} \sin(2tr)H(r)\,\dd r = \int_0^\infty \sin(2tr)H(r)\,\dd r.
$$
We claim that $H\in L^1(0,\infty)$. Indeed,
$$
\int_0^\infty|H(r)|\,\dd r \leq \int_0^\infty \int_{\Sph^{N-1}} |\widehat u_1(r\omega)| |\widehat u_0(r\omega)| \,\dd\omega\,\dd r = \int_{\R^N} \frac{ |\widehat u_1(\xi)| |\widehat u_0(\xi)| }{ |\xi|^{N-1} } \,\dd\xi .
$$
We split the last integral into the regions $|\xi|\leq1$ and $|\xi|\geq1$. On $\{|\xi|\leq1\}$, the boundedness of the Fourier transforms gives
$$
\int_{|\xi|\leq1} \frac{ |\widehat u_1(\xi)| |\widehat u_0(\xi)| }{ |\xi|^{N-1} } \,\dd\xi \leq \norm{u_1}_{L^1} \norm{u_0}_{L^1} \int_{|\xi|\leq1} |\xi|^{-(N-1)} \,\dd\xi = \sigma_{N-1} \norm{u_1}_{L^1} \norm{u_0}_{L^1} .
$$
On $\{|\xi|\geq1\}$, we have $|\xi|^{-(N-1)}\leq1$, and the Cauchy--Schwarz inequality gives
$$
\int_{|\xi|\geq1} \frac{ |\widehat u_1(\xi)| |\widehat u_0(\xi)| }{ |\xi|^{N-1} } \,\dd\xi \leq \norm{\widehat u_1}_{L^2} \norm{\widehat u_0}_{L^2} < \infty.
$$
Therefore $H\in L^1(0,\infty)$, and the one-dimensional Riemann--Lebesgue lemma gives $\mathcal D_{s_c}(t)\to0$.

\smallskip
\noindent{Conclusion.} By \eqref{eq:expansion}, we have
$$
\norm{u(t,\cdot)}_{\dH^{-s_c}}^2 = \frac{1}{(2\pi)^N} \left( \mathcal A_{s_c}(t) + \mathcal B_{s_c}(t) + \mathcal D_{s_c}(t) \right).
$$
Combining the preceding asymptotic formulas, we obtain
$$
\norm{u(t,\cdot)}_{\dH^{-s_c}}^2 = \frac{\sigma_{N-1}P^2}{2(2\pi)^N} \log t + \mathcal C_2 + o(1),
$$
where
$$
\mathcal C_2 = \frac{1}{2(2\pi)^N} \left[ \sigma_{N-1}P^2\log2 + \int_0^\infty \frac{ g(r)-\sigma_{N-1}P^2e^{-r} }{r} \,\dd r \right] + \frac12 \norm{u_0}_{\dH^{-s_c}}^2 .
$$
This proves \eqref{eq:secondorder2} and \eqref{eq:C2}.

Finally, when $N=2$, we have $s_c=0$, $\sigma_1=2\pi$ and $(2\pi)^N=4\pi^2$. Hence $\frac{\sigma_1P^2}{2(2\pi)^2} = \frac{P^2}{4\pi}$ and $\frac{1}{2(2\pi)^2} = \frac{1}{8\pi^2}$. Substituting these values into \eqref{eq:secondorder2}, we obtain \eqref{eq:secondorder2d}.
\end{proof}

\begin{remark}\rm\label{rem:cross-vanishes}
It is worth contrasting the two second-order results. In dimension one, the cross term $\mathcal D_0$ contributes the nonzero amount $\pi PQ$, so that the initial position $u_0$ enters $\mathcal C_1$ both through $\norm{u_0}^2$ and through its total mass $Q$. At critical regularity in dimension $N\ge2$, by contrast, the cross term vanishes in the limit, and $u_0$ enters $\mathcal C_2$ only through $\norm{u_0}_{\dH^{-s_c}}^2$. The reason is purely dimensional. The weight $\abs{\xi}^{-2s_c-1}=\abs{\xi}^{1-N}$ is exactly compensated by the volume element $\abs{\xi}^{N-1}$, so that $\mathcal D_{s_c}$ becomes a plain Fourier sine transform of an $L^1$ function, with no boundary contribution at the origin.
\end{remark}

\begin{remark}\rm\label{rem:weaker-moment}
The moment assumption $u_1\in L^{1,1}(\R^N)$ in Theorem~\ref{th:secondorder2} is used only to verify the Dini condition \eqref{eq:dini} for $g$. Since $\abs{e^{-ix\cdot\xi}-1}\le\min\{2,\abs x\abs\xi\}\le2^{1-\delta}\abs{x}^\delta\abs{\xi}^\delta$ for $0<\delta\le1$, we have $\abs{\widehat u_1(\xi)-P}\le2^{1-\delta}\abs\xi^\delta\int_{\R^N}\abs x^\delta\abs{u_1(x)}\,\dd x$. Hence $\abs{g(r)-g(0^+)}\le Cr^\delta$ for $0<r<1$, and the proof applies verbatim if $u_1\in L^{1,\delta}(\R^N)\cap L^2(\R^N)$ for some $\delta\in(0,1]$.
\end{remark}

\section{Examples}\label{sec:examples}

In this section, we illustrate our results by six explicit examples. All the numerical values below have been checked numerically to at least six significant digits.

\begin{example}[The box in $\R$: an exactly solvable case]
\label{ex:ikehata}
Let $N=1$, $u_0\equiv0$ and $u_1=2\cdot\mathbf 1_{[-1,1]}$, so that $P=4$ and $Q=0$. A direct computation from d'Alembert's formula shows that, for {$t\ge1$}, the solution equals $2$ on $\{\abs{x}\le t-1\}$, is equal to $t+1-\abs{x}$ on $\{t-1<\abs{x}<t+1\}$, and vanishes elsewhere. We infer that
\begin{equation}\label{eq:exact-example}
\norm{u(t,\cdot)}^2=8(t-1)+\frac{16}{3}=8t-\frac83,\qquad {t\ge1} .
\end{equation}
On the other hand, $\hat u_1(\xi)=4\sin\xi/\xi$, so that $P^2/2=8$ and
\[
\mathcal C_1=\frac{1}{4\pi}\int_\R
\frac{16\frac{\sin^2\xi}{\xi^2}-16}{\xi^2}\dd\xi
=\frac{4}{\pi}\int_\R\frac{\sin^2\xi-\xi^2}{\xi^4}\dd\xi .
\]
Two integrations by parts, together with $\int_0^\infty(1-\cos2x)x^{-2}\dd x=\pi$, give $\int_\R(\sin^2\xi-\xi^2)\xi^{-4}\dd\xi=-\tfrac{2\pi}{3}$, and hence $\mathcal C_1=-\tfrac83$. Theorem~\ref{th:secondorder1} therefore predicts $\norm{u(t,\cdot)}^2=8t-\tfrac83+o(1)$, in exact agreement with \eqref{eq:exact-example}. In fact, the $o(1)$ term vanishes identically for {$t\ge1$}. Note that the asymptotic profile $u(t,\cdot)\approx\tfrac P2\mathbf 1_{\{\abs{x}<t\}}$ is precisely the extremizer of the Cauchy--Schwarz inequality {$\abs{\int u(t,x)\dd x}\le\sqrt{2(t+1)}\,\norm{u(t,\cdot)}$}, combined with the conservation law \eqref{eq:massconservation}, as anticipated in Subsection~\ref{ss:mechanism}.
\end{example}

\begin{example}[Gaussian velocity in $\R$]\label{ex:1dgauss}
Let $N=1$, $u_0\equiv0$ and $u_1(x)=e^{-x^2}$, so that $\hat u_1(\xi)=\sqrt\pi\,e^{-\xi^2/4}$ and $P=\sqrt\pi$. Then {$u_1\in L^{1,1}(\R)$}, and the identity $\int_\R\xi^{-2}\bigl(e^{-a\xi^2}-1\bigr)\dd\xi=-2\sqrt{\pi a}$ {(obtained by one integration by parts)} with $a=1/2$ gives
\[
\mathcal C_1=\frac{1}{4\pi}\int_\R\frac{\pi e^{-\xi^2/2}-\pi}{\xi^2}\dd\xi
=\frac{1}{4}\bigl(-2\sqrt{\pi/2}\bigr)=-\frac{\sqrt{2\pi}}{4}.
\]
Consequently,
\[
\norm{u(t,\cdot)}^2_{L^2(\R)}=\frac{\pi}{2}\,t-\frac{\sqrt{2\pi}}{4}+o(1)
\approx \frac\pi2 t-0.626657\ldots
\]
\end{example}

\begin{example}[Poisson velocity in $\R$: $L^1$ but not $L^{1,1}$]
\label{ex:cauchy}
Let $N=1$, $u_0\equiv0$ and $u_1(x)=(1+x^2)^{-1}$, so that $\hat u_1(\xi)=\pi e^{-\abs{\xi}}$ and $P=\pi$. Here $u_1\in L^1\cap L^2$ and {$u_1\in L^{1,\alpha}(\R)$ for every $\alpha<1$}, but $u_1\notin L^{1,1}(\R)$. Hence {neither} the lower bound in \cite[Theorem~1.1]{Ikehata2023} {nor} Theorem~\ref{th:secondorder1} applies, whereas Theorem~\ref{th:profiles} does. In fact, everything can be computed in closed form. Differentiating in $t$ and using $\int_0^\infty\xi^{-1}e^{-b\xi}\sin(2t\xi)\dd\xi=\arctan(2t/b)$ {with $b=2$}, we find, for all $t>0$,
\[
\norm{u(t,\cdot)}^2
=\pi\Bigl[t\arctan t-\tfrac12\log\bigl(1+t^2\bigr)\Bigr]
=\frac{\pi^2}{2}\,t-\pi\log t-\pi+O(t^{-2}).
\]
The first-order term is $\tfrac{P^2}{2}t=\tfrac{\pi^2}{2}t$, in agreement with Corollary~\ref{cor:L2}. The second-order term, however, is $-\pi\log t$, which is {not} a constant. {This shows that the moment hypothesis $u_1\in L^{1,1}(\R)$ in Theorem~\ref{th:secondorder1} cannot be replaced by $u_1\in L^{1,\alpha}(\R)$ for any $\alpha<1$ (see Remark~\ref{rem:moment-sharp}),} and that the family of possible second-order behaviours is richer than \eqref{eq:secondorder1} suggests. It also illustrates comment (b) of Subsection~\ref{ss:results}: {Theorem~\ref{th:profiles} gives the exact leading term for data outside the scope of the lower bound in \eqref{eq:ike1}.}
\end{example}

\begin{example}[Vanishing mass in $\R$: the convergent regime]
\label{ex:twopoisson}
Let $N=1$ and $u_0\equiv0$, fix $a>0$ with $a\ne1$, and take
\[
u_1(x)=\frac1\pi\frac{1}{1+x^2}-\frac1\pi\frac{a}{a^2+x^2},
\qquad\text{so that}\qquad
\hat u_1(\xi)=e^{-\abs{\xi}}-e^{-a\abs{\xi}},\qquad P=0 .
\]
Then $\abs{\hat u_1(\xi)}^2\xi^{-2}$ is bounded near the origin {and decays exponentially at infinity}, so that $u_1\in\dH^{-1}(\R)$ and Theorem~\ref{th:dichotomy} applies. {Note that $u_1(x)\sim\frac{1-a}{\pi}x^{-2}$ as $\abs x\to\infty$, so that $u_1\notin L^{1,1}(\R)$.} We use the identity
$$
\int_0^\infty y^{-2}\sum_jc_je^{-p_jy}\dd y=\sum_jc_jp_j\log p_j ,
$$
valid whenever $\sum_jc_j=\sum_jc_jp_j=0$ {(indeed, the function $G(p):=\int_0^\infty y^{-2}\bigl(e^{-py}-1+py\,\mathbf 1_{(0,1)}(y)\bigr)\dd y$ satisfies $G''(p)=1/p$ for $p>0$, so that $G(p)=p\log p+\alpha p+\beta$ for some constants $\alpha,\beta$, and the two conditions give $\sum_jc_je^{-p_jy}=\sum_jc_j\bigl(e^{-p_jy}-1+p_jy\mathbf 1_{(0,1)}(y)\bigr)$)}, with $(c_j)=(1,-2,1)$ and $(p_j)=(2,1+a,2a)$. We obtain
\[
\lim_{t\to\infty}\norm{u(t,\cdot)}^2
=\frac12\norm{u_1}^2_{\dH^{-1}}
=\frac{1}{\pi}\Bigl[\log2-(1+a)\log(1+a)+a\log(2a)\Bigr] .
\]
The bracket vanishes at $a=1$ (where $u_1\equiv0$) and is positive otherwise. For $a=2$, its value is {$0.169899\ldots$}, so that $\norm{u(t,\cdot)}^2\to0.0540805\ldots$ Note that this example is {not} covered by the lower bound in \eqref{eq:ike1}, which is vacuous when $P=0$.
\end{example}

\begin{example}[Gaussian velocity in $\R^2$: the logarithmic regime to second order]\label{ex:2dgauss}
Let $N=2$, $u_0\equiv0$ and $u_1(x)=e^{-\abs{x}^2}$, so that $\hat u_1(\xi)=\pi e^{-\abs{\xi}^2/4}$, $P=\pi$, and $g(r)=2\pi^3e^{-r^2/2}$ with $g(0^+)=2\pi^3=\sigma_1P^2$. Here $u_1\in L^{1,1}(\R^2)$, so that Theorem~\ref{th:secondorder2} applies. We use the classical evaluation
\[
\int_0^\infty\frac{e^{-r^2/2}-e^{-r}}{r}\dd r=\frac{\gamma+\log2}{2},
\]
which is obtained by truncating both integrals at $r=\varepsilon$ and using $E_1(\varepsilon)=-\gamma-\log\varepsilon+O(\varepsilon)$ {for the exponential integral $E_1(\varepsilon)=\int_\varepsilon^\infty e^{-r}r^{-1}\dd r$ (see \cite[5.1.11]{AS1964})}. Formula \eqref{eq:secondorder2d} then gives
\[
\norm{u(t,\cdot)}^2_{L^2(\R^2)}
=\frac{\pi}{4}\log t+\frac{\pi\bigl(3\log2+\gamma\bigr)}{8}+o(1)
\approx\frac\pi4\log t+{1.043267}\ldots
\]
The leading constant $P^2/(4\pi)=\pi/4$ agrees with Corollary~\ref{cor:L2} and with the heuristic \eqref{eq:heur2}. {This example shows, in a particular but representative case, the exact value of the bounded term that is left undetermined by Lemma~\ref{lem:log}.}
\end{example}

\begin{example}[Gaussian velocity in $\R^3$: boundedness in $L^2$, blowup in $\dH^{-1/2}$]\label{ex:3dgauss}
Let $N=3$, $u_0\equiv0$ and $u_1(x)=e^{-\abs{x}^2}$, so that $\hat u_1(\xi)=\pi^{3/2}e^{-\abs{\xi}^2/4}$ and $P=\pi^{3/2}$. Since $N=3$, we have $s_c=1/2$, and {Theorems~\ref{th:profiles} and~\ref{th:secondorder2}} give simultaneously
\[
\lim_{t\to\infty}\norm{u(t,\cdot)}^2_{L^2(\R^3)}
=\frac12\norm{u_1}^2_{\dH^{-1}}
=\frac{\pi^{3/2}}{4\sqrt2}=0.984350\ldots,
\]
\[
\norm{u(t,\cdot)}^2_{\dH^{-1/2}(\R^3)}=\frac{P^2}{4\pi^2}\log t+{\frac{\pi(3\log2+\gamma)}{8}+o(1)}
=\frac{\pi}{4}\log t+{1.043267\ldots+o(1)}.
\]
{(Here $g(r)=4\pi^4e^{-r^2/2}$, and the constant follows from \eqref{eq:C2} exactly as in Example~\ref{ex:2dgauss}.)} Thus the same solution is bounded in $L^2$ and unbounded in $\dH^{-1/2}$. The infinite-time blowup of \cite{Ikehata2023} has not disappeared in dimension three; it has simply moved to a weaker norm. Compare with Remark~\ref{rem:higher-d}.
\end{example}

\section{Application to local energy decay in two dimensions}\label{sec:LED}

We briefly recall the mechanism, following \cite{Morawetz1961,IkeharaNishihara2005,CharaoIkehata2020,Ikehata2023}. For sufficiently regular real-valued solutions of \eqref{eq:wave}--\eqref{eq:data}, the Morawetz identity in the whole space reads
\begin{equation}\label{eq:morawetz}
tE(t)=\frac{N-1}{2}\int_{\R^N}u_1u_0\dd x
+\int_{\R^N}u_1(x\cdot\nabla u_0)\dd x
-\frac{N-1}{2}F(t)
-\int_{\R^N}u_t(t,x)\,(x\cdot\nabla u(t,x))\dd x,
\end{equation}
with
\[
F(t):=\int_{\R^N}u_t(t,x)u(t,x)\dd x .
\]
{For smooth compactly supported data, \eqref{eq:morawetz} follows by differentiating in time the quantity $\int_{\R^N}u_t(x\cdot\nabla u)\dd x+\frac{N-1}{2}F(t)$ and using the equation together with the identities $\int_{\R^N}\Delta u\,(x\cdot\nabla u)\dd x=\frac{N-2}{2}\norm{\nabla u}^2$ and $\int_{\R^N}u_t\,(x\cdot\nabla u_t)\dd x=-\frac N2\norm{u_t}^2$. One finds that its derivative equals $-E(t)=-E(0)$.} By the Cauchy--Schwarz inequality and the conservation of energy, we have
\begin{equation}\label{eq:F-bound}
\abs{F(t)}\le\norm{u(t,\cdot)}\,\norm{u_t(t,\cdot)}
\le\sqrt{2E(0)}\ \norm{u(t,\cdot)}.
\end{equation}
Hence the whole issue is the control of $M(t)=\norm{u(t,\cdot)}$, which is precisely the object of Theorem~\ref{th:profiles}. For $N=2$, the following midway estimate holds.

\begin{lemma}\label{lem:midway}
Let $N=2$, and let $[u_0,u_1]\in H^1(\R^2)\times L^2(\R^2)$ satisfy
$$
\int_{\R^2} |x| \left( |u_1(x)|^2+|\nabla u_0(x)|^2 \right) \,\dd x < \infty.
$$
Then, for every $R>0$ and every $t>R$,
\begin{equation}\label{eq:midway}
(t-R)E_R(t) \leq K_0+\frac12|F(t)|,
\end{equation}
where $E_R$ is defined in \eqref{eq:localenergy} and
\begin{equation}\label{eq:K0}
K_0 := \int_{\R^2} u_1(x) \bigl(x\cdot\nabla u_0(x)\bigr) \,\dd x + \frac12 \int_{\R^2}u_1(x)u_0(x)\,\dd x + \frac12\int_{\R^2}\abs{x}\left(\abs{u_1(x)}^2+\abs{\nabla u_0(x)}^2\right)\dd x.
\end{equation}
\end{lemma}

\begin{proof}
All the terms in \eqref{eq:midway} depend continuously on the data for the norm $\norm{u_0}_{H^1}+\norm{u_1}+\bigl(\int_{\R^2}\abs x(\abs{u_1}^2+\abs{\nabla u_0}^2)\dd x\bigr)^{1/2}$, for each fixed $t$. By a standard density argument, we may therefore assume that $u_0,u_1\in C_c^\infty(\R^2)$. Then $u$ is smooth and, by finite speed of propagation, $u(t,\cdot)$ has compact support for every $t\ge0$. We set $e(t,x):=\frac12\bigl(u_t(t,x)^2+\abs{\nabla u(t,x)}^2\bigr)$.

Since $E(t)=E(0)$, identity \eqref{eq:morawetz} with $N=2$ can be written as
$$
\int_{\R^2}\Bigl(t\,e(t,x)+u_t(t,x)\,\bigl(x\cdot\nabla u(t,x)\bigr)\Bigr)\dd x=J_0-\frac12F(t),
$$
where
$$
J_0:=\frac12\int_{\R^2}u_1u_0\dd x+\int_{\R^2}u_1(x\cdot\nabla u_0)\dd x .
$$
The pointwise bound $\abs{u_t\,(x\cdot\nabla u)}\le\abs x\abs{u_t}\abs{\nabla u}\le\abs{x}\,e$ yields
\begin{equation}\label{eq:weighted-step}
\int_{\R^2}\bigl(t-\abs{x}\bigr)e(t,x)\,\dd x\le J_0+\frac12\abs{F(t)} .
\end{equation}
We split the integral on the left-hand side into the regions $\{\abs x\le R\}$, $\{R<\abs x<t\}$ and $\{\abs x\ge t\}$. The integrand is at least $(t-R)e$ on the first region and is nonnegative on the second one. We deduce that the left-hand side of \eqref{eq:weighted-step} is bounded from below by $(t-R)E_R(t)-\Phi(t)$, where
$$
\Phi(t):=\int_{\R^2}\bigl(\abs{x}-t\bigr)_+\,e(t,x)\,\dd x .
$$
We claim that $\Phi$ is nonincreasing. Indeed, $\partial_te=\operatorname{div}(u_t\nabla u)$, and the function $\psi(t,x):=(\abs x-t)_+$ is Lipschitz, with $\partial_t\psi=-\mathbf 1_{\{\abs x>t\}}$ and $\nabla\psi=\frac{x}{\abs x}\mathbf 1_{\{\abs x>t\}}$ almost everywhere. Hence, for almost every $t$,
$$
\Phi'(t)=-\int_{\{\abs x>t\}}e\,\dd x-\int_{\{\abs x>t\}}u_t\,\frac{x}{\abs x}\cdot\nabla u\,\dd x\le-\int_{\{\abs x>t\}}\bigl(e-\abs{u_t}\abs{\nabla u}\bigr)\dd x\le0 .
$$
Consequently, $\Phi(t)\le\Phi(0)=\frac12\int_{\R^2}\abs x\bigl(\abs{u_1}^2+\abs{\nabla u_0}^2\bigr)\dd x$. Inserting this bound into \eqref{eq:weighted-step}, we obtain $(t-R)E_R(t)\le J_0+\Phi(0)+\frac12\abs{F(t)}=K_0+\frac12\abs{F(t)}$, which is \eqref{eq:midway}.
\end{proof}

\begin{remark}\rm\label{rem:midway-ikehata}
Lemma~\ref{lem:midway} is a variant of the midway estimate \cite[Proposition~4.1]{Ikehata2023}, which relies on \cite[(2.13)]{IkeharaNishihara2005} and \cite[(3.5)]{CharaoIkehata2020}. In \cite{Ikehata2023}, the local energy is defined without the factor $\frac12$, and the constant of the midway estimate is
$$
\int_{\R^2}u_1(x\cdot\nabla u_0)\dd x+\frac12\int_{\R^2}u_1u_0\dd x+E(0),
$$
so that the total energy $E(0)$ appears in place of the weighted energy in \eqref{eq:K0}. The weighted energy cannot be dispensed with. Indeed, replace the data by $u_0(x/\lambda)$ and $\lambda^{-1}u_1(x/\lambda)$, and $(R,t)$ by $(\lambda R,\lambda t)$, with $\lambda>0$. Then $(t-R)E_R(t)$, $F(t)$, $J_0$ and the weighted energy are all multiplied by $\lambda$, whereas $E(0)$ is unchanged. Hence, if an estimate of the form $(t-R)E_R(t)\le J_0+cE(0)+\frac12\abs{F(t)}$, with a constant $c\ge0$, held for all data, then letting $\lambda\to\infty$ would give $(t-R)E_R(t)\le J_0+\frac12\abs{F(t)}$ for all $t>R>0$. Taking $t=2R$ and letting $R\to0$, we would infer that $0\le J_0+\frac12\abs{F(0)}$. This fails for $u_0=\varphi$ and $u_1=-\bigl(x\cdot\nabla\varphi+\frac12\varphi\bigr)$ with $\varphi(x)=e^{-\abs x^2}$. For these data, we have $J_0=-\norm{x\cdot\nabla\varphi+\frac12\varphi}^2=-\frac{5\pi}{8}$ and $F(0)=\frac12\norm{\varphi}^2=\frac{\pi}{4}$, and hence $J_0+\frac12\abs{F(0)}=-\frac{\pi}{2}<0$. This argument does not depend on the normalization of $E_R$. It does not affect the decay rate $t^{-1}\sqrt{\log t}$ in \cite[Theorem~4.1]{Ikehata2023} either, since only the value of the constant changes. Here, however, the precise value of $K_0$ matters, because it enters \eqref{eq:LEDbounded}. This is why we have included the short proof of Lemma~\ref{lem:midway}.
\end{remark}

\begin{proof}[Proof of Theorem \ref{th:LED}]
{By assumption, $u_0\in L^2(\R^2)$ and $u_1\in L^1(\R^2)\cap L^2(\R^2)$.} Therefore, Corollary~\ref{cor:L2} gives
\[
\norm{u(t,\cdot)}=\frac{\abs{P}}{2\sqrt\pi}\sqrt{\log t}+o\bigl(\sqrt{\log t}\bigr),
\qquad t\to\infty ,
\]
where the right-hand side is understood as $o(\sqrt{\log t})$ when $P=0$. Combining this with \eqref{eq:F-bound} and {Lemma~\ref{lem:midway}}, we get
\[
E_R(t)\le\frac{K_0}{t-R}
+\frac{\sqrt{2E(0)}\,\abs{P}}{4\sqrt{\pi}}\cdot
\frac{\sqrt{\log t}}{t-R}+o\left(\frac{\sqrt{\log t}}{t-R}\right) .
\]
We multiply by $t/\sqrt{\log t}$ and use $t/(t-R)\to1$ and $1/\sqrt{\log t}\to0$. Since $\sqrt2/(4\sqrt\pi)=1/(2\sqrt{2\pi})$, we obtain \eqref{eq:LEDsharp}.

If, in addition, $u_1\in\dH^{-1}(\R^2)$, Theorem~\ref{th:dichotomy} gives
$$
\norm{u(t,\cdot)}\longrightarrow\Bigl[\tfrac12\bigl(\norm{u_1}^2_{\dH^{-1}}+\norm{u_0}^2\bigr)\Bigr]^{1/2}.
$$
Using again \eqref{eq:F-bound} and \eqref{eq:midway}, we deduce that $E_R(t)=O(1/t)$, and then, since $t/(t-R)\to1$,
$$
\limsup_{t\to\infty}tE_R(t) \leq K_0 + \frac{\sqrt{2E(0)}}{2}\cdot\frac{1}{\sqrt2} \left( \norm{u_1}_{\dH^{-1}}^2+\norm{u_0}^2 \right)^{1/2},
$$
which is \eqref{eq:LEDbounded}. Finally, let $u_1\in L^{1,1}(\R^2)$ with $P=0$. Then
$$
|\widehat u_1(\xi)| = |\widehat u_1(\xi)-\widehat u_1(0)| \leq |\xi| \int_{\R^2}|x||u_1(x)|\,\dd x,
$$
so that $|\xi|^{-2}|\widehat u_1(\xi)|^2$ is bounded on $\{|\xi|\leq1\}$. On $\{|\xi|\geq1\}$, we have $|\xi|^{-2}|\widehat u_1(\xi)|^2 \leq |\widehat u_1(\xi)|^2$, which is integrable by Plancherel's theorem. Hence $u_1\in\dH^{-1}(\R^2)$, and \eqref{eq:LEDbounded} applies.
\end{proof}

\begin{remark}\rm
Estimate \eqref{eq:LEDsharp} refines \cite[Theorem~4.1]{Ikehata2023} within the same multiplier framework. The leading constant is explicit and depends only on $|P|\sqrt{E(0)}$, rather than on the whole quantity $I_{0,2}$. This constant is sharp at the level of the argument, since it is obtained by inserting the exact leading asymptotics $\norm{u(t,\cdot)} = \frac{|P|}{2\sqrt{\pi}}\sqrt{\log t} + o\bigl(\sqrt{\log t}\bigr)$ into the midway estimate \eqref{eq:midway}. Estimate \eqref{eq:LEDbounded} further shows that the logarithmic loss disappears whenever $u_1\in\dH^{-1}(\R^2)$. This condition is more flexible than moment conditions of the type $u_1\in L^{1,1}(\R^2)$, $P=0$, which appear in earlier works such as \cite{CharaoIkehata2020,Ikehata2019}. Indeed, the latter implies $u_1\in\dH^{-1}(\R^2)$, whereas membership in $\dH^{-1}(\R^2)$ does not require a finite first absolute moment.
{For instance, for $a>0$ with $a\neq1$, the two-dimensional analogue of Example~\ref{ex:twopoisson},
$$
u_1(x)=\frac1{2\pi}\Bigl[(1+\abs x^2)^{-3/2}-a\,(a^2+\abs x^2)^{-3/2}\Bigr],
\qquad
\widehat u_1(\xi)=e^{-\abs\xi}-e^{-a\abs\xi},
$$
belongs to $\dH^{-1}(\R^2)$, satisfies $\int_{\R^2}\abs x\abs{u_1}^2\dd x<\infty$, and has $P=0$. However, $u_1(x)\sim\frac{1-a}{2\pi}\abs x^{-3}$ as $\abs x\to\infty$, so that $u_1\notin L^{1,1}(\R^2)$.}
The present argument does not decide whether $t^{-1}\sqrt{\log t}$ is the optimal decay rate of the local energy $E_R(t)$ itself. It only proves that this is the rate produced by the multiplier estimate combined with the sharp asymptotics of $\norm{u(t,\cdot)}$. The optimality of the rate for $E_R(t)$ remains open{, and, as R.~Ikehata pointed out to us, it seems to be a difficult question}; see Problem~\ref{pb:LEDsharp}.
\end{remark}

\section{Further remarks and future directions} \label{sec:open}

In this section, we discuss further applications of our method and several directions for future research.

\begin{problem}[Second order without moments]\label{pb:nomoments}
Example~\ref{ex:cauchy} shows that, when $u_1\in L^1\setminus L^{1,1}$, the second-order term in dimension one need not be constant. For the Poisson kernel, it equals $-\pi\log t$. Describe the whole family of possible second-order terms. Concretely, suppose {(say with $u_0\equiv0$)} that $|\hat u_1(\xi)|^2-P^2\sim-c|\xi|^\beta$ as $\xi\to0$, where $0<\beta<2$ and $c>0$. For $0<\beta<1$, the scaling identity \eqref{eq:scaling} suggests that $\norm{u(t,\cdot)}^2=\frac{P^2}{2}t-c_\beta t^{1-\beta} +o(t^{1-\beta})$, with $c_\beta=\frac{c}{2\pi}\int_{\R}\sin^2|\eta|\,|\eta|^{\beta-2}\dd\eta$. At the borderline exponent $\beta=1$, one expects a logarithmic correction, which is exactly what Example~\ref{ex:cauchy} exhibits. For $1<\beta<2$, the function $\xi^{-2}(|\hat u_1(\xi)|^2-P^2)$ is locally integrable, so that the second-order term is constant, while a correction of order $t^{1-\beta}$ may appear at the next order. Making this trichotomy precise under minimal assumptions on the remainder, and computing all the corresponding constants, seems within reach of the present methods.
\end{problem}

\begin{problem}[Anomalous rates and regular variation]\label{pb:RV}
Theorem~\ref{th:profiles} describes the case $\hat u_1(0)=P\ne0$. Suppose instead that $P=0$ but $u_1\notin\dH^{-1}$, with a spherical average behaving regularly at the origin, say
\[
g(r)=\int_{\Sph^{N-1}}\abs{\hat u_1(r\omega)}^2\dd\omega
= r^{2\beta}\,\ell(1/r)\,\bigl(1+o(1)\bigr),\qquad r\to0^+,
\]
for some $\beta\ge0$ and some function $\ell$ slowly varying in the sense of Karamata {\cite{BGT1987}}. The scaling identity \eqref{eq:scaling} suggests that
\[
\norm{u(t,\cdot)}^2\ \sim\ c_{N,\beta}\ t^{\,2-N-2\beta}\,\ell(t),
\qquad 0\le\beta<\tfrac{2-N}{2},
\]
with an explicit constant $c_{N,\beta}$, and a logarithmic correction at the critical value $\beta=(2-N)/2$. {This range of $\beta$ is nonempty only for $N=1$; for $N=2$, only the critical value $\beta=0$ occurs.} In dimension one, this would produce the full continuum of rates $t^{\alpha}$, $0\le\alpha\le\tfrac12$, {up to slowly varying factors}, for $\norm{u(t,\cdot)}$. These rates interpolate between Theorem~\ref{th:dichotomy} and Corollary~\ref{cor:L2}, and they would show that Remark~\ref{rem:norate} is sharp. Making this rigorous, and characterizing the velocities $u_1$ that produce a prescribed rate, would give a complete Karamata-type classification of infinite-time $L^2$-blowup for the free wave equation. The main technical point is that $g$ is a spherical average, so that regular variation must be imposed on $g$ rather than on $\hat u_1$ itself. The one-dimensional case, where the sphere consists of two points, should be settled first.
\end{problem}

\begin{problem}[Optimality of the local energy rate]\label{pb:LEDsharp}
Is the rate $t^{-1}\sqrt{\log t}$ in \eqref{eq:LEDsharp} attained by some datum with $P\ne0$? Or is the genuine decay of $E_R$ faster, the factor $\sqrt{\log t}$ being an artifact of the multiplier method? A matching lower bound for $E_R$ in dimension two seems to be unknown. Note that \eqref{eq:midway} loses information twice, once in the Cauchy--Schwarz step \eqref{eq:F-bound} and once by discarding the oscillation of $F$. Since Theorem~\ref{th:profiles} now identifies $\norm{u(t,\cdot)}$ exactly at leading order, the remaining gap lies entirely in the correlation between $u$ and $u_t$, which is itself a natural object of study.
{The question concerns data that are not compactly supported. Indeed, let $u_0,u_1\in C_c^\infty(\R^2)$ vanish outside $\{\abs{x}\le a\}$, and let $\abs{x}\le R$ and $t>R+a$. Poisson's formula then reads $u=\partial_t(W*u_0)+W*u_1$, with
$$
W*f(t,x)=\frac1{2\pi}\int_{\R^2}\frac{f(z)}{\sqrt{t^2-\abs{x-z}^2}}\,\dd z,
$$
and the expansion $(t^2-\abs{x-z}^2)^{-1/2}=t^{-1}+O(t^{-3})$, which may be differentiated term by term, gives $u_t(t,x)=-\frac{P}{2\pi}t^{-2}+O(t^{-3})$ and $\nabla u(t,x)=O(t^{-3})$, uniformly in $\abs{x}\le R$. Hence $E_R(t)=\frac{P^2R^2}{8\pi}\,t^{-4}+O(t^{-5})$, which is much smaller than $t^{-1}\sqrt{\log t}$. The relevant class is therefore that of data decaying slowly near infinity, as in \cite{IkeharaNishihara2005,IkehataSobukawa2007}, and determining the optimal rate in the class of Theorem~\ref{th:LED} seems to be difficult.}
\end{problem}

\begin{problem}[Other models]\label{pb:models}
The proofs above use only the explicit multiplier in \eqref{eq:fourier-solution} and the continuity of $\hat u_1$ at $\xi=0$. They should therefore adapt to other dispersive models. For the Klein--Gordon equation $u_{tt}-\Delta u+m^2u=0$, the multiplier
$$
\frac{\sin\bigl(t\sqrt{\abs{\xi}^2+m^2}\bigr)}{\sqrt{\abs{\xi}^2+m^2}}
$$
is bounded. Boundedness of $M$ is then expected in all dimensions, with limit
$$
\frac12\Bigl(\norm{(m^2-\Delta)^{-1/2}u_1}^2+\norm{u_0}^2\Bigr).
$$ For the plate (beam) equation $u_{tt}+\Delta^2u=0$, the multiplier is $\sin(t\abs{\xi}^2)/\abs{\xi}^2$, so that the critical index becomes $s_c=(N-4)/2$ and the exceptional dimensions are $N\le4$. For the fractional wave equation $u_{tt}+(-\Delta)^\sigma u=0$, the critical index is $s_c=(N-2\sigma)/2$. Strongly damped wave equations, for which analogous questions were the original motivation in \cite{Ikehata2014,IkehataOnodera2017}, are also natural candidates. {Optimal growth rates of the $L^2$-norm, with unspecified constants, are already known for the plate equation \cite{Ikehata2024plate} and, in dimension one, for the fractional wave equation \cite{IkedaZhao2024}; they agree with the critical indices above.} Determining the exact constants in each of these models, in particular the analogue of $\kappa_{N,s}$, is a natural continuation of the present {paper}.
\end{problem}

\begin{problem}[Exterior domains]\label{pb:exterior}
Does Theorem~\ref{th:dichotomy} admit an analogue in an exterior domain with Dirichlet boundary conditions? {The $L^2$-behaviour of solutions in exterior domains was studied in \cite{IkehataMatsuyama2002}.} The difficulty is that the Fourier transform is no longer available, and the quantity $P$ is no longer conserved in the sense of \eqref{eq:massconservation}. A resolvent-based low-frequency expansion should replace \eqref{eq:fourier-solution}. One may expect the role of $\dH^{-1}(\R^N)$ to be played by the domain of the inverse square root of the Dirichlet Laplacian, or by an appropriate low-frequency range space. Determining the correct formulation is part of the problem. This is also the setting in which the local energy decay question of \cite{Morawetz1961} was originally posed.
\end{problem}

\section*{Acknowledgements}
{The authors are grateful to Professor Ryo Ikehata for reading a preliminary version of this manuscript and for his valuable comments, in particular on the historical context of the infinite-time $L^2$-growth problem, on the origin of the identity \eqref{eq:massconservation}, and on the optimality of the local energy decay rate.}


\end{document}